\documentclass{article}
\usepackage[utf8]{inputenc}
\usepackage{turnstile}
\usepackage[toc,page]{appendix}
\usepackage{amsmath,amssymb,amsthm,lmodern,geometry}
\usepackage[dvipsnames]{xcolor}
\usepackage{centernot}
\usepackage{authblk}
\usepackage{comment}
\usepackage{bussproofs}
\usepackage{nicefrac}
\usepackage{proof,color}
\usepackage{leftindex}

\usepackage[textwidth=2cm]{todonotes}
\usepackage{graphicx}

\newcommand{\myexists}{\ensuremath\exists\kern-.45em\exists}
\newcommand{\myforall}{\ensuremath\forall\kern-.66em\forall\,}
\newcommand{\myvee}{\ensuremath\vee\kern-.98em\vee\,}
\newcommand{\myneg}{\ensuremath\neg\kern-.8em\neg\,}
\newcommand{\mywedge}{\ensuremath\wedge\kern-.98em\wedge\,}

	\theoremstyle{definition}
		\newtheorem{remark}{Remark}

        \newtheorem*{thmmm}{{Theorem 3}}

        \newtheorem*{cortwo}{{Corollary 2}}
	\theoremstyle{plain}
	\newtheorem{dfn}{Definition}

			\newtheorem{thm}{Theorem}
		\newtheorem{lemma}{Lemma}
		\newtheorem{prop}{Proposition}
		\newtheorem{corollary}{Corollary}

        \theoremstyle{definition}
\newtheorem{exmp}{Example}
        \newcommand{\T}{\mathrm{Tr}}
        
        \newcommand{\lra}{\leftrightarrow}

        \newcommand*\subdot[1]{\oalign{$#1$\cr\hfil.\hfil}}%
        \newcommand{\mc}[1]{\mathcal{#1}}
        \newcommand{\lnat}{\mathcal{L}_\mathbb{N}}
        \newcommand{\lt}{\mc{L}_\T}
        \newcommand{\mrm}{\mathrm}
        \newcommand{\vphi}{\varphi}
        \newcommand{\beq}{\begin{equation}}
        \newcommand{\eeq}{\end{equation}}

        \newcommand{\ra}{\rightarrow}
        \newcommand{\hatt}{\mathbf{HAT}}
        \newcommand{\Ra}{\Rightarrow}
        
        \newcommand{\nat}{\mathbb{N}}
        \newcommand{\sth}{\;|\;}

        \newcommand{\uph}{ \upharpoonright}
        
        \newcommand{\jcsv}{\mc{J}_\mrm{csv}}
        \newcommand{\VdashM}{\Vdash^{\mc{M}}_\mrm{csv}}
        
        \newcommand{\corn}[1]{\ulcorner #1\urcorner}
        
        \newcommand{\lipc}{\mc{L}_\mrm{IPC}}
        \newcommand{\lbox}{\mc{L}_\Box}
        \newcommand{\Lra}{\Leftrightarrow}

\usepackage{natbib}
\title{Supervaluations, truth, and intutionistic logic}
\author[1]{Pablo Dopico}
\affil[1]{\small Fachbereich Philosophie, Universität Konstanz, pablo.dopico@uni-konstanz.de}
\date{}

\begin{document}

\maketitle

\begin{abstract} 
The supervaluationist approach to fixed-point semantics is, arguably, the most celebrated and studied competitor to the Strong Kleene approach within Kripkean truth. In this paper, we show how to obtain supervaluationist fixed-point theories of truth for intuitionistic logic. In particular, we show how to do supervaluations over Kripke structures for intuitionistic logic, and we obtain the corresponding theories of truth, both semantic and axiomatic. Furthermore, we technically compare our approach with previous attempts to formalize Kripkean supervaluational theories over possible-world semantics. Finally, we show that, unlike in the case of classical logic, a fixed-point theory over intuitionistic logic can be given that is both supervaluational and highly compositional in nature. 
\end{abstract}

\section{Introduction}

In his \cite{kripke_1975}, Saul Kripke devised a celebrated theory of truth, based on a fixed-point construction, that delivered a partial but fully transparent truth predicate. In other words, this theory meets the intuition that a sentence $\vphi$ is true if and only if $\vphi$ is the case. Mathematically, the construction on which the theory is based runs as follows: one starts with a set of sentences granted to be true, and then applies an operator that collects the consequences of those sentences on some partial logic. By employing a monotone operator, the well-known theory of positive inductive definitions guarantees the existence of fixed-points, including a minimal one, i.e., a fixed-point contained in any other fixed-point. If the partial logic is adequately defined, such fixed-points will meet the desired transparency condition. 

Historically, the partial logic most often used has been Strong Kleene (SK) logic. Its resulting fixed-point theory has motivated the famous (classical) axiomatic theory of Kripke-Feferman, KF (\cite{feferman_1991, cantini_1989}), as well as the (arguably) more faithful non-classical theory known as Partial Kripke-Feferman, PKF \citep{halbach_horsten_2006}. But Kripke immediately noticed that the theory based on SK logic would fail to declare true many truths of classical logic. For example, `$\lambda\rightarrow\lambda$', for $\lambda$ the Liar sentence, fails to be in the fixed-points of the SK construction. 

It is as a remedy to this failure that supervaluational fixed-point theories appear. Thus, Kripke's idea was to use as the partial logic a form of supervaluational logic in the style of \cite{van-fraassen_1966}. The original motivation behind supervaluationism is straightforward: to produce a logic that allows for truth-value gaps---and, more specifically, those engendered by non-denoting terms---while preserving all the truths of classical logic. Similarly, in the context of Kripke's theory, supervaluational logics become a way to recover all those classical validities that are lost in SK truth. And while the SK approach remains the most popular one, supervaluational truth has by now been studied in detail (\cite{cantini_1990, kremer-urquhart_2008, meadows_2013}) and is attracting a good deal of interest in recent years (\cite{stern_2018, incurvati_schloder_2023, stern_2025}). Consequently, we nowadays have not only a full picture of \textit{semantic} supervaluational truth in the different conceptions proposed by Kripke, but also a fairly thorough understanding of \textit{axiomatic} supervaluational truth as given by \cite{cantini_1990}, and, more recently, \cite{dopico_hayashi_2024}. 

There is a caveat, though. All these contributions---with the notable exception of \cite{stern_2025}---concern truth achieved by supervaluating over classical models. That is, one supervaluates by looking at the set of sentences satisfied by all classical models meeting certain conditions. But supervaluation is, after all, a method, and nothing prevents us from applying such method to an alternative class of semantic structures. Therefore, the question which motivates this paper is: what does supervaluation over intuitionistic structures look like? 

There are at least a couple of very relevant motivations for addressing this question. In the first place, and as mentioned, supervaluational truth emerged as a way to recover the classical validities that SK truth just takes as gappy. But classical truth is by no means universally accepted. A defender of constructive mathematics will be wary to accept all classical validities, and yet might be on board with accepting intuitionistic tautologies. Hence, supervaluational truth over intuitionistic logic might be more suitable as a formal theory of truth for the constructive mathematician.\footnote{Note that SK truth would still fail to declare true many intuitionistic validities, such as $\lambda\ra\lambda$, which the constructive mathematician might find true, and hence there is still a rationale for her to seek some form of supervaluationism.} This motivation is in line with the few advances in the field of intuitionistic truth that have appeared in the literature during the last years---mainly, \cite{leigh_rathjen_2012, leigh_2013,  zicchetti_fischer_2024}. 

In the second place, there is a technical interest: since the presentation of Kripke models for intuitionistic logic is rather different from the way classical models work, it is by no means immediate how to generalize supervaluations to the case of intuitionistic structures. In this sense, this paper follows somehow the spirit of previous work where intensional notions (such as truth, but also necessity) have been presented as a predicate in a setting based on possible-worlds---see, e.g., \cite{halbach_leitgeb_welch_2003, halbach_welch_2009, nicolai_2018a}.\footnote{For an axiomatic treatment, see \cite{stern_2014b, stern_2014}.} 

Finally, a further motivation stems from the limitations of supervaluationist truth over intuitionistic logic. In particular, its incompatibility with compositional properties for disjunctions and existentially quantified sentences. Given that intuitionistic arithmetic, unlike its classical counterpart, possesses the disjunction and existential properties---that is, a disjunction will be provable if and only if a disjunct is, and similarly for the existentially quantified statements---, it is reasonable to hope that supervaluationist theories over intuitionistic logic could retain compositionality for these connectives.

The paper is divided as follows. In the next section, we briefly introduce our notation for the remainder of the paper. In section \ref{semantics}, we offer the basic fixed-point construction for supervaluations over intuitionitic logic, working with the standard Kripke models for Heyting Arithmetic. In section \ref{axiomatics}, we present the axiomatic theory that corresponds to our semantic construction. Further, section \ref{fixed-frame approach} compares our approach with what we dubbed the \textit{fixed-frame} approach, the one found in preceding work on Kripkean truth over possible-world semantics. Finally, in section \ref{section:compositional supervaluations} we show that obtaining a theory that is both supervaluational and highly compositional is indeed possible: by modifying the forcing relation for the Kripke structures, we produce a supervaluational theory where the truth predicate distributes over disjunction and the existential quantifier.  

\section{Notational preliminaries}

Following the standard convention, our intuitionistic languages have as logical symbols  $\bot, \ra, \vee,$ $\wedge, \forall$ and $\exists$, as well as brackets. As concerns the arithmetical coding for our theories of truth, we follow mostly the conventions in \cite[ch.5]{halbach_2014}. Our base language $\lnat$ is a definitional extension of the language of Heyting Arithmetic ($\mrm{HA}$) with finitely many function symbols for primitive recursive functions. We assume a standard formalization of the syntax of first-order languages, including $\lnat$ itself---see e.g. \cite{hajek_pudlak_2017}. As is well-known, this can be carried out in HA. $\lnat$ is assumed to contain a finite set of function symbols that will stand for certain primitive recursive operations. For example, $\subdot \ra$ is a symbol for the primitive recursive function that, when inputted the code of two formulae $\vphi$ and $\psi$, yields the code of the conditional $\vphi\ra \psi$. The same applies to $\subdot \vee, \subdot \forall, \subdot \wedge, \subdot \exists$. 
We assume a function symbol for the substitution function, and write $x (t/v)$ for the result of substituting $v$ with $t$ in $x$; $\ulcorner \varphi(\dot x)\urcorner$ abbreviates $\ulcorner \varphi (v)\urcorner (\mrm{num}(x)/\corn{v}) $, for $x$ a term variable (and provided $\varphi$ only has one free variable). Furthermore, we write $t^\circ$ for the result of applying to a term $t$ the evaluation function (which outputs the value of the inputted term). Note that this is an abbreviation for a formula, and not a symbol of the language. In addition, for a closed term $t$ of the language, we write $t^\nat$ for the interpretation of this term in the standard (classical) model of arithmetic $\nat$.

  Given a language $\mc{L}$, we write  $\mathrm{Sent}_{\mc{L}}(x)$ for the formula representing the set of all sentences of $\mc{L}$. We use $\mrm{SENT}_\mc{L}$ to denote the set itself.  We will occasionally omit reference to $\mc{L}$ when this is clear from the context.  Here, we are basically interested in the cases in which $\mc{L}$ is $\lnat$ or $\lt:=\lnat\cup \{\T\}$. The expression $\Bar{n}$ stands for the numeral of the number $n$ (although we omit the bar for specific numbers). We write $\# \varphi$ for the code or Gödel number of $\varphi$, and $\ulcorner \varphi\urcorner$ for the numeral of that code. We sometimes write $\corn{\vphi(\subdot t)}$ as an abbreviation for $\corn{\vphi(x)}(t/\corn{x})$.

$\mrm{HAT}$ is the theory formulated in the language $\lt$ consisting of the axioms of $\mrm{HA}$ (i.e., the axioms of PA over intuitionistic predicate logic---see, e.g., \cite[ch.3]{troelstra_van-dalen_1988}) with induction extended to the whole language $\lt$. $\mrm{HAT}$ will be our background syntax theory.

Finally, if $T$, $S$ are theories, we write $\vert T\vert \geqq \vert S\vert$ to indicate that $T$ is at least as proof-theoretically strong as $S$, i.e., proves all the $\lnat$ theorems of $S$. Similarly, $\vert T\vert \equiv \vert S\vert$ is defined as $\vert T\vert \geqq \vert S\vert$ and $\vert S\vert \geqq \vert T\vert$.

\section{Semantics}\label{semantics}

In order to define intuitionistic supervaluational fixed-point semantics, we will make use of Kripke structures for intuitionistic logic. For the sake of simplification---and since different kinds of Kripke structures will give rise to different theories of truth---we will work with what Leigh and Rathjen \citeyearpar{leigh_rathjen_2012} call intuitionistic Kripke $\omega$-structures for the language $\lt$ . 

\subsection{The structures}

\begin{dfn} \label{persistent FO int structures}
A \emph{first-order intuitionistic Kripke} $\omega$\emph{-structure} $\mc{M}$ for $\lt$ is a quadruple 

\begin{equation*}
\langle W_\mc{M}, \leq_\mc{M}, \mc{D}_\mc{M}, \mc{I}_\mc{M}\rangle
\end{equation*}
such that:

\begin{itemize}
    \item $W_\mc{M}$ is a set of worlds.
    \item $\leq_\mc{M}$ is a partial order on $W_\mc{M}$.
    \item $\mc{D}_{\mc{M}}= W_\mc{M} \times \{\omega\} $
    \item $\mc{I}_\mc{M}\subseteq W_\mc{M} \times \mrm{SENT}_{\lt} $
\end{itemize}

\noindent Moreover, we say that a first-order intuitionistic Kripke $\omega$-structure  $\mc{M}=\langle W_\mc{M}, \leq_\mc{M}, \mc{D}_\mc{M}, \mc{I}_\mc{M}\rangle$ is \emph{persistent} iff $\mc{I}_\mc{M}$ satisfies the hereditary condition, i.e., for all $u,v\in W_\mc{M}$, $u\leq_\mc{M} v$ and $\langle u,x\rangle \in \mc{I}_\mc{M}$ implies $\langle v,x\rangle \in \mc{I}_\mc{M}$.
\end{dfn}

\noindent Remember that $\mrm{SENT}_{\lt}$ stands for the set of Gödel numbers of sentences of $\lt$ and is therefore a subset of $\omega$. Given a structure $\mc{M}:=\langle W_\mc{M}, \leq_\mc{M}, \mc{D}_\mc{M}, \mc{I}_\mc{M}\rangle$, we will occasionally call the triple $\langle W_\mc{M}, \leq_\mc{M}, \mc{D}_\mc{M}\rangle$ the \textit{frame} of $\mc{M}$. Also, we write $\mc{I}_u$ for $\{ x: \langle u,x\rangle \in \mc{I}_\mc{M}\}$. Since we call the members of $W_\mc{M}$ \textit{worlds}, $\mc{D}_\mc{M}$ is to be thought as the domain-assigning function which assigns $\omega$ to every world; and $\mc{I}_u$ is the \textit{interpretation of truth at world u}. With this terminology, the hereditary condition is best rewritten as: for all $u,v\in W_\mc{M}$, $u\leq_\mc{M} v$ implies $\mc{I}_u\subseteq \mc{I}_v$. We also note that, if unspecified, it is assumed that, when we speak of a structure $\mc{M}$, it corresponds to the tuple $\langle W_\mc{M}, \leq_\mc{M}, \mc{D}_\mc{M}, \mc{I}_\mc{M}\rangle$.

When we restrict a persistent first-order intuitionistic Kripke-$\omega$ structure $\mc{M}$ to the structure generated by one of its worlds $w$ and its accessibility relation, we speak of the \textit{truncated} structure $\mc{M}_w$:

\begin{dfn}\label{truncated structure}
Given a persistent first-order intuitionistic Kripke $\omega$-structure $\mc{M}$ and a world $w\in W_\mc{M}$, the \emph{truncated structure} $\mc{M}_w$ is the quadruple $\langle W_{\mc{M}_w}, \leq_{\mc{M}_w}, \mc{D}_{\mc{M}_w}, \mc{I}_{\mc{M}_w}\rangle$, where: 
\begin{itemize}
    \item $W_{\mc{M}_w}=\{ w'\sth w'\geq_\mc{M} w\}$
    \item $\leq_{\mc{M}_w}=\leq_\mc{M}\uph W_{\mc{M}_w}$
    \item $\mc{D}_{\mc{M}_w}= \mc{D}_\mc{M}\uph W_{\mc{M}_w}$
    \item $\mc{I}_{\mc{M}_w}=\mc{I}_{\mc{M}}\uph W_{\mc{M}_w}$
\end{itemize}
\end{dfn}

\noindent Now we define a forcing relation $w \Vdash A$, for $A$ a formula of $\lt$, as follows: 

\begin{dfn}\label{forcing relation}
Given a persistent first-order intuitionistic Kripke $\omega$-structure $\mc{M}$ and a world $w\in W_\mc{M}$, we define recursively the relation $w \Vdash A$ as follows: 

\begin{enumerate}
    \item $w\nVdash \bot$.
    \item $w\Vdash R(t_1,...,t_n)$ iff $R(t_1, ...,t_n) $ is satisfied in the standard model $\nat$, where $R$ is an n-ary predicate symbol of $\lnat$ for a primitive recursive relation and $t_1,...,t_n$ are terms of $\lnat$.
    \item $w\Vdash \T(s)$ iff there is a sentence $\psi$ of $\lt$ such that $s^\nat=\#\psi$ and $\#\psi\in \mc{I}_w$.
    \item $w\Vdash A\wedge B$ iff $w\Vdash A$ and $w\Vdash B$.
    \item $w\Vdash A\vee B$ iff $w\Vdash A$ or $w\Vdash B$.
    \item $w\Vdash A\rightarrow B$ iff for every $w'\geq w$, $w'\Vdash A$ implies $w'\Vdash B$.
    \item $w\Vdash \exists x A(x)$ iff there exists some $n\in \omega$ such that $w\Vdash A(\bar n)$.
    \item $w\Vdash \forall x A(x)$ iff for all $w'\geq w$ and all $n\in \omega$ we have $w'\Vdash A(\bar n)$.
\end{enumerate}
\end{dfn}
\noindent We write $w\Vdash_\mc{M} A$ when it is deemed appropriate to remark that the underlying structure is $\mc{M}$.

\begin{dfn}
 Given a persistent first-order intuitionistic Kripke $\omega$-structure $\mc{M}$ and a formula $A$ of $\lt$, we write $\mc{M}\vDash A$ iff $w\Vdash A$ for all $w\in W_\mc{M}$.
\end{dfn}

\noindent Because of the hereditary condition on interpretations of the truth predicate, the following holds: 

\begin{prop}
All persistent first-order intuitionistic Kripke $\omega$-structures are hereditary first-order intuitionistic Kripke structures. That is, for any such structure $\mc{M}$ and worlds $w, w'\in W_\mc{M}$, and any formula $\vphi$ of $\lt$: if $w\leq_\mc{M} w'$ and $w\Vdash \vphi$, then $w'\Vdash \vphi$.
\end{prop}

\noindent It is known that intuitionistic first-order predicate calculus is both sound and complete with respect to first-order hereditary intuitionistic Kripke structures, and hence sound with respect to persistent first-order intuitionistic Kripke $\omega$-structures.\footnote{See \cite[Ch.5]{troelstra_van-dalen_1988}.} Since $\omega$-structures satisfy the Peano axioms, it is also evident that:

\begin{prop}
If $\mc{M}$ is a persistent first-order intuitionistic Kripke $\omega$-structure, $\mc{M}\vDash \hatt$. Therefore, if $\hatt\vdash \vphi$, $\mc{M}\vDash \vphi$.
\end{prop}

\noindent We also remark a feature of truncated structures that is straightforward to establish. The proof just relies on the definition of the forcing clauses:

\begin{prop}
Given a persistent first-order intuitionistic Kripke $\omega$-structure $\mc{M}$, a world $w\in W_\mc{M}$, and a formula $\vphi$ of $\lt$, for any $w'\geq_\mc{M}w$:
\begin{equation*}
    w'\Vdash_{\mc{M}_w}\vphi \Lra w'\Vdash_{\mc{M}}\vphi
\end{equation*}
\end{prop}

\noindent This result is assumed in several places throughout this paper. 

In the remaining of this section, $\mc{M}, \mc{N}, \mc{M}'$, etc. will range over persistent first-order intuitionistic Kripke $\omega$-structures.

\subsection{The semantic construction}

The simplest approach for a semantic, supervaluational theory in classical logic based on the fixed-point approach is the following. Given an interpretation for the truth predicate, $X$, one considers all models of $\lt$ such 
that their arithmetical part is just $\nat$ and whose interpretation of the truth predicate includes $X$. In other words: one considers all classical structures where the interpretation for the truth predicate \textit{extends} the one given.  (Sometimes, additional conditions are imposed; we will return to that later on.) Thus, writing $\langle \nat, S\rangle$ for the classical model for $\lt$ over the standard model $\nat$ where we let $S$ be the interpretation of $\T$, the classical approach is formalized as:

\begin{equation}\label{basic classical supervaluation}
\langle \nat, X\rangle\vDash_{\mrm{sv}} A \text{ iff for all } X'\supseteq X, \langle \nat,  X'\rangle \vDash A
\end{equation}

\noindent Here, $\mc{M} \vDash_{\mrm{sv}} A$ is to be read as ``$A$ is sv-satisfied by $\mc{M}$'', and each model of the form $\langle \nat,  X'\rangle $ is a way of extending $\langle \nat, X\rangle$.

However, since intuitionistic structures differ from classical ones, we need to translate the classical supervaluationist framework in the the most faithful way possible. In particular, there are two differences between the classical and the intuitionistic case, concerning what a extension of the interpretation for the truth predicate is, that we see as crucial. These rely on an understanding of the interpretation for the truth predicate which equates the latter with the relation $\mc{I}_\mc{M}$, where $\mc{M}$ is some first-order intuitionistic Kripke $\omega$-structure. First, the relation, and with it the interpretation for $\T$, can be extended by adding to the structure \textit{more worlds} over which the function is defined. That is, by enlarging a given structure with more worlds (and, possibly, more elements to the accessibility relation), and defining the interpretation of truth at those worlds---in such a way that respects the requirements for intuitionistic Kripke $\omega$-structures---we are producing an extended truth interpretation. Formally, we start with a structure $\mc{M}$ whose truth interpretation function is $\mc{I}_\mc{M}$. Then, in considering another another structure $\mc{N}$ which possesses all the worlds in $\mc{M}$ and additional ones, and whose interpretation of truth for all worlds already in $\mc{M}$ is that of $\mc{M}$, $\mc{I}_\mc{N}$ genuinely extends $\mc{I}_\mc{M}$, i.e., $\mc{I}_\mc{M}\subseteq \mc{I}_\mc{N}$. Second, the relation, and with it the interpretation for $\T$, might be extended by assigning \textit{larger sets of sentences} to some of the worlds, and keeping the rest equal. That is, a truth interpretation $\mc{I}_\mc{N}$ might extend another $\mc{I}_\mc{M}$ by having the same worlds, but assigning a larger interpretation of truth for some of those worlds---provided it respects the hereditary conditions. In that case, clearly, we also have $\mc{I}_\mc{M}\subseteq \mc{I}_\mc{N}$. In addition, and combining these two points, an interpretation $\mc{I}_\mc{N}$ will also extend another when $\mc{N}$ contains all the worlds in $\mc{M}$ plus additional ones, \textit{and} the interpretation of truth at every shared world is equal or greater in $\mc{N}$ than in $\mc{M}$.

In sum, a structure $\mc{M}$ extending the truth interpretation of another structure $\mc{N}$ means that: i) $\mc{M}$ includes at least all the worlds contained in $\mc{N}$, and; ii) in every world contained in both $\mc{M}$ and $\mc{N}$, the interpretation of $\T$ at $\mc{M}$ extends the interpretation of $\T$ at $\mc{N}$. This being said, there are different ways of conceiving an extended truth interpretation---we will touch on a previously explored such conception in section \ref{fixed-frame approach}. For now, we stick to the one just propounded. A simple formal rendering of it could be the following:

\begin{dfn}[Naive interpretation extension]
Given structures $\mc{M}$ and $\mc{M}'$ \emph{naively extends} $\mc{I}_{\mc{M}}$, in symbols $\mc{I}_{\mc{M}}\sqsubseteq \mc{I}_{\mc{M}'}$ iff 

\begin{itemize}
    \item $\mc{I}_{\mc{M}}\subseteq \mc{I}_{\mc{M}'}$
    \item $\leq_{\mc{M}'}\uph W_\mc{M}=\leq_\mc{M}$
\end{itemize}
\end{dfn}

\noindent Note that $\mc{I}_{\mc{M}}\sqsubseteq \mc{I}_{\mc{M}'}$ implies $W_\mc{M}\subseteq W_{\mc{M}'}$. And, of course, $\mc{I}_\mc{M}\sqsubseteq \mc{I}_\mc{M}$ for all $\mc{M}$; any interpretation extends itself. So we could really think of one interpretation  for a structure extending another in this naive sense when the latter is a substructure of the former.

Nonetheless, this is not good enough. For what if we have two structures where worlds have different labels but, modulo label adjustment, one interpretation extends the other? Consider the following simple example: we have a one-world structure $\mc{M}$, with $W_\mc{M}=\{w\}$ and $\mc{I}_\mc{M}(w)=X$; and another one-world structure $\mc{N}$, with $W_\mc{N}=\{v\}$ and $\mc{I}_\mc{N}(v)=X'\supseteq X$. If, in order to define supervaluational truth, one wants to take into account \textit{all} structures whose interpretation is in some sense an extension of  $\mc{I}_\mc{M}$, it seems that one would like to consider $\mc{N}$. After all, $\mc{N}$ is structurally the same as $\mc{M}$, and the interpretation of truth at the only world in $\mc{N}$ extends that given to the only world in $\mc{M}$. Thus, we introduce the notion of \textit{interpretation embeddings}: 

\begin{dfn}[Embedding-interpretation function]\label{embedding function}
Given structures $\mc{M}$ and $\mc{M}'$, and a function $f:W_\mc{M}\longrightarrow W_{\mc{M}'}$, $f$ is an embedding-interpretation function (or EI function) from $\mc{M}$ into $\mc{M}'$ iff all of the following hold:

\begin{enumerate}
    \item $f$ is injective
    \item for all $w_0, w_1\in W_\mc{M}$: $w_0\leq_\mc{M}w_1\Leftrightarrow f(w_0)\leq_{\mc{M}'} f(w_1)$
    \item for all $w\in W_\mc{M}$ and sentences $\vphi$ of $\lt$: $\langle w, \#\vphi\rangle \in \mc{I}_\mc{M} \Rightarrow \langle f(w), \#\vphi\rangle \in \mc{I}_{\mc{M}'}$
\end{enumerate}
\end{dfn}

\begin{dfn}[Interpretation embedding]\label{interpretation embedding}
Given structures $\mc{M}:=\langle W_\mc{M}, \leq_\mc{M}, \mc{D}_\mc{M}, \mc{I}_\mc{M}\rangle$ and $\mc{M}':=\langle W_\mc{M'}, \leq_\mc{M'}, \mc{D}_\mc{M'}, \mc{I}_\mc{M'}\rangle$, we say that $\mc{I}_{\mc{M}}$ \emph{is embeddable in} $\mc{I}_{\mc{M}'}$, in symbols $\mc{I}_{\mc{M}}\leqq \mc{I}_{\mc{M}'}$, iff there is a function $f:W_\mc{M}\longrightarrow W_{\mc{M}'}$ which is an embedding-interpretation function from $\mc{M}$ into $\mc{M}'$. 
\end{dfn}

\noindent When we want to signify that $\mc{I}_{\mc{M}}$ is embeddable in $\mc{I}_{\mc{M}'}$ via the function $f$, we write $\mc{I}_{\mc{M}}\leqq_f \mc{I}_{\mc{M}'}$. The existence or not of an EI function between two structures is then really what determines whether the interpretation for the truth predicate of one of them extends the interpretation of the other. We might sometimes be a little sloppy and simply say that $\mc{M}$ is embeddable in $\mc{M'}$. This being settled, we can define the supervaluational forcing relation, akin to the supervaluational satisfaction relation of the classical case:

\begin{dfn}\label{svi-forcing}
Let $\mc{M}$ be a structure, and let $w\in W_\mc{M}$. We define the relation $w\Vdash^{\mc{M}}_\mrm{svi}\vphi$ as follows:

\begin{center}
    $w\Vdash^{\mc{M}}_\mrm{svi}\vphi: \Leftrightarrow \forall \mc{M}'\forall f(\mc{I}_\mc{M}\leqq_f \mc{I}_{\mc{M}'} \Rightarrow f(w)\Vdash_{\mc{M}'} \vphi)$.
\end{center}
\end{dfn}

\noindent $w\Vdash^{\mc{M}}_\mrm{svi}\vphi $ is to be read as ``$w$ svi-forces $\vphi$''. $w\in W_\mc{M}$ svi-forces $\vphi$ when, for any structure into which $\mc{M}$ is embeddable via some function $f$, the image of $w$ under the embedding function forces $\vphi$. Thus, the notion of svi-forcing is tied to a world $w$. The antecedent in the definiendum (in this case, $\mc{I}_\mc{M}\leqq_f \mc{I}_{\mc{M}'}$) is called the admissibility requirement. We refer to the forcing scheme defined by the svi-forcing relation as the scheme SVI.

Let us now give the informal picture. Forcing in Kripke models is a relation between a world and a formula, but determined by a structure, an accessibility relation, and a truth interpretation. Hence, in supervaluating, we range over alternative structures, accessibility relations, and truth interpretations, thus considering what could have been somewhat extended while respecting what is, so to say, `determinately true'---i.e., the original interpretation of the truth predicate at a given world.

Next, we define a Kripke-jump; the latter is no longer relative to a world:

\begin{dfn} \label{jump operator}
Let $X\subseteq \mrm{SENT}_{\lt}$. The Kripke-jump operator over the svi-forcing relation is a function $\mc{J}_\mrm{svi}:\mc{P}(\mrm{SENT}_{\lt})\longrightarrow \mc{P}(\mrm{SENT}_{\lt})$ defined as follows::

\begin{center}
$\mc{J}_\mrm{svi}(X):=\{\# \vphi \sth \forall \mc{M}, \forall w\in W_\mc{M}(\mc{I}_{\mc{M}} (w)=X \Rightarrow w\Vdash^{\mc{M}}_\mrm{svi}\vphi)\}$  
\end{center}
\end{dfn}

\noindent By analogy with the classical case, we can abbreviate the set condition on the right-hand-side, i.e. $\forall \mc{M}, \forall w\in W_\mc{M}(\mc{I}_{\mc{M}} (w)=X \Rightarrow w\Vdash^{\mc{M}}_\mrm{svi}\vphi)$, by $X\vDash_\mrm{svi}\vphi$. Roughly put, then, $X\vDash_\mrm{svi}\vphi$ amounts to what is supervaluationally forced at all worlds with truth interpretation equal to $X$, abstracting away from the particularities of a world in a structure. 

As is to be expected: 

\begin{prop}\label{svi-intuitionistic-validities} Let $X\subseteq \mrm{SENT}_{\lt}$. Then $X\vDash_\mrm{svi} \vphi$, for $\vphi$ any intuitionistically valid formula of $\lt$.
\end{prop}

\noindent On the other hand, $X\vDash_\mrm{svi} \vphi$ will not always hold when $\vphi$ is a classical but not intuitionistic validity. For example:

\begin{exmp}
Let $X\subsetneq \mrm{SENT}_{\lt}$ be such that there is a sentence $\vphi$ of $\lt$ for which $\#\vphi\notin X$ and $\#\neg\vphi \notin X$.
Then $X\nvDash_\mrm{svi} \T\corn{\vphi}\vee \neg \T\corn{\vphi}$.
\end{exmp}
\begin{proof}
Take any structure $\mc{M}$ and worlds $w_0, w_1\in W_\mc{M}$ such that:

\begin{itemize}
    \item $\mc{I}_{\mc{M}} (w_0)=X$, 
    \item $\mc{I}_{\mc{M}}(w_1)=X\cup\{\#\vphi\}$
    \item $w_0\leq_\mc{M} w_1$
\end{itemize}

Consider the identity embedding. Since $w\nVdash_{\mc{M}} \T\corn{\vphi}\vee \neg \T\corn{\vphi}$, then $w\nVdash^{\mc{M}}_\mrm{svi} \T\corn{\vphi}\vee \neg \T\corn{\vphi}$, whence $X\nvDash_\mrm{svi} \T\corn{\vphi}\vee \neg \T\corn{\vphi}$ follows.
\end{proof}

\noindent The result in Proposition \ref{svi-intuitionistic-validities} can in fact be refined in the following way:

\begin{prop}\label{svi-intersection}
Let $X\subseteq \mrm{SENT}_{\lt} $. Then $\forall \mc{M},w(\mc{I}_\mc{M}(w)=X\Ra w\Vdash_\mc{M}\vphi)$ implies $\#\vphi\in\mc{J}_\mrm{svi}(X)$. 
\end{prop}

\begin{proof}
Pick some arbitrary $\mc{M}$ and $w$ s.t. $\mc{I}_\mc{M}(w)=X$. It suffices to show $w\Vdash^\mc{M}_\mrm{svi}\vphi$. So take a structure  $\mc{N}$ and function $f$ s.t. $\mc{I}_\mc{M}\leqq_f \mc{I}_\mc{N}$, and show $f(w)\Vdash_\mc{N}\vphi$. Since $\mc{I}_\mc{N}(f(w))=X'\supseteq X$, take the truncated structure $\mc{N}_{f(w)}$. Now, create a new structure $\mc{N}'$ by adding a root $n'$ that accesses $f(w)$, and such that $\mc{I}_{\mc{N}'}(n')=X$. It's easy to see that this structure is well-defined and, by assumption, $n'\Vdash_{\mc{N}'}\vphi$. So, by hereditariness, $f(w)\Vdash_{\mc{N}'}\vphi$, which entails $f(w)\Vdash_{\mc{N}_{f(w)}}\vphi$. Finally, the latter entails $f(w)\Vdash_\mc{N}\vphi$. 
\end{proof}

\noindent That is, the SVI-jump over $X$ really captures whatever is satisfied by all worlds with truth interpretation equal to $X$.

\subsection{Fixed-points and fixed-point models}\label{fixed-points and fixed-point models}

Let us now investigate the SVI-jump and, in particular, the fixed-points that it gives rise to. First, we show that said fixed-points exist. 

\begin{prop}\label{monotonicity svi operator}
The following holds:
\begin{itemize}
    \item[i)] $\mc{J}_\mrm{svi}(\cdot)$ is monotonic, i.e., $X\subseteq Y\Ra \mc{J}_\mrm{svi}(X)\subseteq \mc{J}_\mrm{svi}(Y)$. Therefore, there are fixed points of $\mc{J}_\mrm{svi}(\cdot)$.
    \item[ii)] For all $\mc{M}$ and worlds $w\in W_\mc{M}$, there exists a structure $\mc{M}'$ satisfying the admissibility requirements. Therefore, there are non-degenerate fixed points of $\mc{J}_\mrm{svi}(\cdot)$.
\end{itemize}
\end{prop}

\begin{proof}
For i): Let $\#\vphi\notin \mc{J}_\mrm{svi}(Y)$. Then fix $\mc{M}$ and $w\in W_\mc{M}$ such that $\mc{I}_\mc{M}(w)=Y$ and fix $\mc{M}'$ such that $\mc{I}_\mc{M}\leqq_f \mc{I}_{\mc{M}'}$ and $f(w)\nVdash_{\mc{M}'}\vphi$. Now, let $\mc{M}_0$ be the truncated structure $\mc{M}_{f(w)}$, and define a new structure $\mc{M}'_0$ which is exactly like $\mc{M}_0$ but $\mc{I}_{\mc{M}'_0}(f(w))=X$. Then $\mc{I}_{\mc{M}'_0}\leqq_\mrm{id} \mc{I}_{\mc{M}_0}$ (where $\mrm{id}$ is the identity function) and $f(w)\nVdash_{\mc{M}_0}\vphi$, so $\#\vphi\notin \mc{J}_\mrm{svi}(X)$.

For ii): just take $\mc{M}$ itself. 
\end{proof}

\noindent Note that there is an alternative way of defining the Kripke jump:

\begin{dfn}\

\begin{center}
$\mc{J}_\mrm{svi'}(X):=\{\# \vphi \sth \forall \mc{M}, \forall w\in W_\mc{M}(\mc{I}_{\mc{M}} (w)\supseteq X \Rightarrow w\Vdash^{\mc{M}}_\mrm{svi}\vphi)\}$  
\end{center}
\end{dfn}

\begin{prop}
The following holds:
\begin{itemize}
    \item[i)] $\mc{J}_\mrm{svi'}(\cdot)$ is monotonic, i.e., $X\subseteq Y\Ra \mc{J}_\mrm{svi'}(X)\subseteq \mc{J}_\mrm{svi'}(Y)$. Therefore, there are fixed points of $\mc{J}_\mrm{svi'}(\cdot)$.
    \item[ii)] For all structures $\mc{M}$ and worlds $w\in W_\mc{M}$, there exists a structure $\mc{M}'$ satisfying the admissibility requirements. Therefore, there are non-degenerate fixed points of $\mc{J}_\mrm{svi'}(\cdot)$.
\end{itemize}
\end{prop}
\begin{proof}
For i): for any structure $\mc{M}$, world $w\in W_\mc{M}$, $\mc{I}_{\mc{M}} (w)\supseteq Y$ implies $\mc{I}_{\mc{M}} (w)\supseteq X$ when $X\subseteq Y$. Therefore, if $\#\vphi \notin \mc{J}_\mrm{svi'}(Y)$, the structure $\mc{M}'$ and world $w'$ such that $\mc{I}_{\mc{M}'}(w') \supseteq Y$ and $w'\nVdash^{\mc{M}}_\mrm{svi}\vphi$ is a counterexample for $\#\vphi \in \mc{J}_\mrm{svi'}(X)$. 
For ii): as before. 
\end{proof}

\noindent Nonetheless, in what follows, we will stick to the jump $\mc{J}_\mrm{svi}$, as it will make the proofs in section \ref{section:compositional supervaluations} more immediate. Once we know there are fixed points of the operator, we can define fixed-point structures in the following way: 

\begin{dfn}
A fixed-point structure of the scheme $\mrm{SVI}$ is a persistent first-order intuitionistic Kripke $\omega$-structure $\mc{M}:=\langle W_\mc{M}, \leq_\mc{M}, \mc{D}_\mc{M}, \mc{I}_\mc{M}\rangle$ such that for all $w\in W_\mc{M}$, there is an $X\subseteq \mrm{SENT}_{\lt} $ with $X=\mc{J}_\mrm{svi}(X)$ and $\mc{I}_\mc{M}(w)= X$.
\end{dfn}

\noindent In short, a fixed-point structure is a structure of the kind we have been considering where the extension of $\T$ at any world is a fixed point of SVI.

The main advantage of the Kripkean approach to truth is that it yields models which are transparent with respect to the truth predicate. Thus, the fixed-point construction with the Strong Kleene schema gives us interpretations $(X^+, X^-)$ (i.e., extension and antiextension) such that 
\begin{equation*}
    \langle \nat, X^+, X^-\rangle \vDash_\mrm{sk} \vphi \Leftrightarrow \langle \nat, X^+, X^-\rangle\vDash_\mrm{sk} \T\corn\vphi. 
 \end{equation*}
\noindent Similarly, the fixed-point supervaluationist approach, as sketched in (\ref{basic classical supervaluation}) above, gives an interpretation $X$ such that 
\begin{equation*}
    \langle \nat, X\rangle \vDash_\mrm{sv} \vphi \Leftrightarrow \langle \nat, X \rangle \vDash_\mrm{sv} \T\corn\vphi. 
 \end{equation*}
In our intuitionistic approach, transparent models can also be given on the basis of fixed-points of the operator $\mc{J}_\mrm{svi}$. In particular, we can provide transparent models for every fixed-point of $\mc{J}_\mrm{svi}$ based on the one-world frame. This is to be contrasted with the possibility of providing transparent models over arbitrary frames, which we could not establish---cf. section \ref{fixed-frame approach}.

As a notational add-on, for $\mc{M}:=\langle W_\mc{M}, \leq_\mc{M}, \mc{D}_\mc{M}, \mc{I}_\mc{M}\rangle$ a structure, we write $\mc{M}\vDash_\mrm{svi}\vphi$ when $w\Vdash^\mc{M}_\mrm{svi}\vphi$ holds for all $w\in W_\mc{M}$.
\begin{prop}\label{transparent models}
Let $X\subseteq \mrm{SENT}_{\lt}$ be such that $X=\mc{J}_\mrm{svi}(X)$. Let $\mc{M}:=\langle W_\mc{M}, \leq_\mc{M}, \mc{D}_\mc{M}, \mc{I}_\mc{M}\rangle$ be a persistent first-order intuitionistic Kripke $\omega$-structure given by:

\begin{itemize}
    \item $W_\mc{M}=\{w\}$.
    \item $\leq_\mc{M}=\{\langle w,w\rangle\}$.
    \item $\mc{I}_\mc{M}=\{\langle w, \#\vphi \rangle \sth \#\vphi \in X\} $.
\end{itemize}

\noindent Then, $w\Vdash^\mc{M}_\mrm{svi}\vphi \Leftrightarrow w\Vdash^\mc{M}_\mrm{svi}\T\corn\vphi$. Therefore, $\mc{M}\vDash_\mrm{svi}\vphi \Leftrightarrow \mc{M}\vDash_\mrm{svi}\T\corn\vphi$. 
\end{prop}
\begin{proof}
The second claim follows by definition from the first; we prove the latter.\\
Left-to-right: assume $w\Vdash^\mc{M}_\mrm{svi}\vphi$. We want to show $\#\vphi \in X$, i.e.,
\begin{equation*}
    \forall \mc{M}'\forall w'(\mc{I}(w')=X\Ra w'\Vdash^\mc{M'}_\mrm{svi}\vphi)
\end{equation*}
\noindent So take arbitrary $\mc{M'}$ and $w'\in W_\mc{M'}$ with $\mc{I}(w')=X$. In order to prove $w'\Vdash_\mrm{svi}\vphi$, we take $\mc{N}$ with $\mc{I}_\mc{M'}\leqq_f\mc{I}_\mc{N}$, and we want to show $f(w')\Vdash_\mc{N} \vphi$. But now, since $\mc{M}$ is a one-world structure, the function that $g$ that sends $w$ to $w'$ is an EI function between $\mc{M}$ and $\mc{M'}$, and so $\mc{I}_\mc{M}\leqq_g \mc{I}_\mc{M'}$. Composing $g$ and $f$ we get a function $h$ such that $\mc{I}_\mc{M}\leqq_h \mc{I}_\mc{N}$, with $h(w)=f(w')$. From $w\Vdash^\mc{M}_\mrm{svi}\vphi$, we obtain $h(w)\Vdash_\mc{N} \vphi$, i.e., $f(w')\Vdash_\mc{N} \vphi$, as desired. 

Right-to-left: $w\Vdash^\mc{M}_\mrm{svi}\T\corn\vphi$ implies $\#\vphi\in X$. Hence, since $X=\mc{J}_\mrm{sv}(X)$, $\forall \mc{M}'\forall w'(\mc{I}(w')=X\Ra w'\Vdash^\mc{M'}_\mrm{svi}\vphi)$. Let $\mc{M}'$ be $\mc{M}$ and $w'$ be $w$, it follows that $w\Vdash^\mc{M}_\mrm{svi}\vphi$, as desired. 
\end{proof}

\noindent Isolated, then, worlds can exhibit a transparent behaviour when coupled with a fixed-point as their interpretation of truth. Things change if we consider larger structures. It is easy to see where the proof for Proposition \ref{transparent models} will fail if we considered more complex, many-worlds frames. While the move from  $w\Vdash^\mc{M}_\mrm{svi}\T\corn\vphi$ to $w\Vdash^\mc{M}_\mrm{svi}\vphi$ would still go through, the argument for the other direction would break: the Kripke models $\mc{M}'$ over which the jump operator $\mc{J}_\mrm{sv}$ ranges might be structurally very different from the original model $\mc{M}$ where $\vphi$ is svi-forced, hence impeding that $\mc{I}_\mc{M}$ be embedded in $\mc{I}_\mc{M}'$. In other words: since our jump ranges over all Kripke models where some world's interpretation of truth is a fixed-point, knowing that $w\Vdash^\mc{M}_\mrm{svi}\vphi$ does not suffice, as it does not guarantee that $\vphi$ is supervaluationally forced in models based on different frames. 

In spite of this, fixed-points themselves \textit{are} transparent: 

\begin{prop}
Let $X\subseteq \mrm{SENT}_{\lt}$ be such that $X=\mc{J}_\mrm{svi}(X)$. Then, $\#\vphi \in X \Lra \#\T\corn\vphi\in X$. 
\end{prop}
\begin{proof}
Left-to-right: assume $\#\vphi\in X$ and take $\mc{M}$, $w$ such that $\mc{I}_\mc{M}(w)=X$. Since $\#\vphi\in X$, $f(w)\Vdash_\mc{N} \T\corn\vphi$ for any $\mc{N}, f$ such that $\mc{I}_\mc{M}\leqq_f\mc{I}_\mc{N}$, hence $w\Vdash_\mrm{svi}^\mc{M}\T\corn\vphi$. Since $\mc{M}, w$ were arbitrary, $\#\T\corn\vphi\in \mc{J}_\mrm{svi}(X)=X$.\\  
Right-to-left: assume $\#\T\corn\vphi\in X$. Take $\mc{M}, w$ such that $\mc{I}_\mc{M}(w)=X$. Since $\#\T\corn\vphi\in X=\mc{J}_\mrm{svi}(X)$, $w\Vdash^\mc{M}_\mrm{svi}\T\corn\vphi$, which is only possible if $\#\vphi \in X$.
\end{proof}

\subsection{Alternative supervaluation schemes}

The classical approach formulated in (\ref{basic classical supervaluation}) above is, as we mentioned, the simplest supervaluational scheme, which is sometimes known as the SV-scheme. This basic approach considers all possible truth extensions. However, it is far from being the only approach possible. In the literature there are, in fact, three more popular schemes,\footnote{See, e.g., \cite[160-161]{mcgee_1991} and \cite{stern_2018}.} corresponding to additional conditions, or admissibility requirements, that we impose on the extensions:

\begin{align}
\label{eq:svb} X\vDash_\mrm{vb} \varphi &\text{ iff } \forall X' (X'\supseteq X \, \&\,  X' \cap X^- =\varnothing \Rightarrow X' \vDash \varphi)\\
   \label{eq:svc} X\vDash_\mrm{vc} \varphi &\text{ iff } \forall X' (X'\supseteq X\, \&\,  X' \cap X'^-=\varnothing \Rightarrow X' \vDash \varphi)\\
  \label{eq:svd}   X\vDash_\mrm{mc} \varphi &\text{ iff } \forall X' (X'\supseteq X\, \&\, \mrm{MCX}(X') \Rightarrow X' \vDash \varphi)
\end{align}

\noindent Here, $\mrm{MCX}(X)$ stands for `$X$ is a maximally consistent set of sentences of $\lt$', i.e., for any sentence $\vphi$ of $\lt$, either $\#\vphi\in X$ or $\#\neg\vphi\in X$, but not both. Moreover, given a set of sentences $X$, $X^-$ is the set of sentences $\{\#\vphi \sth \#\neg\vphi\in X\}$. In all the cases shown, the corresponding jump-operators are monotonic, hence they have fixed-points. Moreover, as long as one starts applying the operator to consistent sets of sentences, the fixed-points will be non-degenerate.\footnote{In slightly more detail: in starting with a consistent set, one guarantees that there will be truth extensions meeting the admissibility requirements. One way (but not the only) to prove this, which all of VB, VC and MC meet, is via the \textit{compactness} of the admissibility requirement $\Phi$ with respect to the set of consistent interpretations. Write $X:=\{S \subseteq \mrm{SENT}_{\lt}\sth S\cap S^-=\varnothing \}$ and $\Phi^S:=\{S'\subseteq \mrm{SENT}_{\lt} \sth S\subseteq S'\,\&\,\Phi(S')\}$. For $Y\subseteq X$, let $\Phi(Y):=\{\Phi^S \sth S\in Y\}$. Then $\Phi$ is \emph{compact on $X$} iff for all $Y\subseteq X$:
$$\text{if }\Phi^{S_1}\cap\ldots\cap\Phi^{S_n}\neq\varnothing\text{ for all }S_1,\ldots S_n\in Y\text{ and all }n\in\omega,\text{ then }\bigcap\Phi(Y)\neq\varnothing.$$
\noindent We are thankful to Johannes Stern for having clarified this issue.\label{footnote_1}} 

All of this is then reproducible in our intuitionistic setting. Given $\mc{M}$, $w\in W_\mc{M}$, write $\mc{I}^-_{\mc{M}}(w)$ for the set $\{ \#\neg\vphi: \langle w,\#\vphi\rangle \in \mc{I}_\mc{M}\}$. Then:

\begin{dfn}
Let $\mc{M}$ be a structure, and let $w\in W_\mc{M}$. We define the following supervaluational forcing relations:
\begin{align*}
    & w\Vdash^{\mc{M}}_{vbi}\vphi:  \Leftrightarrow \forall \mc{M}', f(\mc{I}_\mc{M}\leqq_f\mc{I}_{\mc{M}'} \;\& \; \forall w'\geq f(w) (\mc{I}_{\mc{M}'}(w')\cap \mc{I}^-_{\mc{M}}(w) =\varnothing) \Rightarrow w'\Vdash_{\mc{M}'} \vphi)\\
    & w\Vdash^{\mc{M}}_{vci}\vphi:  \Leftrightarrow \forall \mc{M}', f(\mc{I}_\mc{M}\leqq_f \mc{I}_{\mc{M}'} \;\& \; \forall w'\geq f(w) (\mc{I}_{\mc{M}'}(w')\cap \mc{I}^-_{\mc{M}'}(w') =\varnothing) \Rightarrow w'\Vdash_{\mc{M}'} \vphi) \\
    & w\Vdash^{\mc{M}}_{mci}\vphi: \Leftrightarrow \forall \mc{M}', f(\mc{I}_\mc{M}\leqq_f \mc{I}_{\mc{M}'} \;\& \; \forall w'\geq f(w) (\mrm{MCX}(\mc{I}_{\mc{M}'}(w')))\Rightarrow w'\Vdash_{\mc{M}'} \vphi)
\end{align*}
\end{dfn}
\noindent As with $\Vdash_\mrm{svi}$, the antecedents of the conditionals in the definitions above are known as the admissibility requirements (or conditions) for the corresponding schemes. For the sake of brevity, we write $\mrm{Ad}_1 (\mc{M}, \mc{M}', f)$, $\mrm{Ad}_2 (\mc{M}, \mc{M}', f)$ and $\mrm{Ad}_3 (\mc{M}, \mc{M}', f)$ for the admissibility requirements of vbi-, vci- and mci-forcing. Thus, for example, $\mc{I}_\mc{M}\leqq_f \mc{I}_{\mc{M}'} \;\& \; \forall w'\geq f(w) (\mc{I}_{\mc{M}'}(w')\cap \mc{I}^-_{\mc{M}}(w) =\varnothing)$ abbreviates as $\mrm{Ad}_1 (\mc{M}, \mc{M}', f)$. As before, the schemes defined by the vbi-, vci-, and mci-forcing relations will be known as, respectively, VBI, VCI, and MBI. 

It is essential that the admissibility conditions of the schemes defined above quantify over all worlds accessible from the image of the world we consider---in the case above, over all worlds accessible from $f(w)$. The reason has to do with the purpose of these conditions. They can be seen as a way of setting requirements on the interpretations to be considered, with the aim of imposing features for the \textit{internal} theory of truth, i.e., for the behaviour of the truth predicate \textit{inside} the fixed-points. For example: the scheme VC, (\ref{eq:svc}) above, imposes that the truth predicate be consistent \textit{inside} the fixed-points. These features, in all the cases we consider, take the form of a negation or conditional when formalized. In the case of VC, for instance, the feature is formalized as the sentence $\#\neg (\T t\wedge \T \subdot \neg t)$ (for any arbitrary closed term $t$), which is present in all fixed-points of the scheme. But the forcing relation for the conditional and the negation in the intuitionistic structures we consider is global (or rather, upward-looking), hence, in order to preserve these features, the admissibility condition must also be upward-looking and not local. That is, precisely, what we achieve by quantifying over all worlds accessible from the image of the world we start with.

The corresponding jumps $\mc{J}_\mrm{vbi}, \mc{J}_\mrm{vci}, \mc{J}_\mrm{mci} $ are exactly like in Definition \ref{jump operator}, using the relevant supervaluational-forcing relations. These operators are all monotonic, and thus give rise to fixed-points. These arguments for monotonicity are all replicas of the proof in Proposition \ref{monotonicity svi operator}.i. For the existence of non-degenerate fixed-points of $\mc{J}_{mci}$, note that, when the starting set of codes of sentences does not overlap with its antiextension (i.e., when, for all $\lt$-sentences $\vphi$, either of $\#\vphi$ or $\#\neg\vphi$ is not in the set), then the extra admissibility conditions on the interpretations can always be fulfilled---either by the initial interpretation itself, in the case of VBI and VCI; or by a maximally consistent extension of the initial interpretation in the case of MCI (see also footnote \ref{footnote_1}). It is also easy to notice that that: 

\begin{prop}\label{alternative schemes} Let $X$ range over consistent set of codes of sentences. Then: 

\begin{enumerate}
     \item[i)] Let $X=\mc{J}_\mrm{vbi}(X)$. Let $t^\nat\in X$ (resp., $\#\neg\psi\in X$, for $\#\psi=t^\nat$). Then $\#\neg\T \subdot\neg t\in X$ (resp. $\#\neg\T t\in X$).
    \item[ii)] Let $X=\mc{J}_\mrm{vci}(X)$. Then $\#\neg (\T t\wedge \T \subdot \neg t)\in X$ for every term $t$.
    \item[iii))] Let $X=\mc{J}_\mrm{mci}(X)$. Then $\#\neg \T t\lra \T \subdot \neg t \in X$ for every term $t$.
\end{enumerate}
\end{prop}

\begin{proof}
For i): For the first part: let $t^\nat = \#\vphi$, take any model $\mc{N}$ and a world $v$ such that $\mc{I}_\mc{N}(v)=X$, and show $v\Vdash^\mc{N}_\mrm{svi}\neg \T \corn{\neg \vphi}$. Consider some model $\mc{N}'$ into which $\mc{N}$ embeds via an EI $f$ and such that $\mrm{Ad}_1 (\mc{N}, \mc{N}', f)$. Since $\#\vphi \in X=\mc{I}_\mc{N}(w)$, then $\#\neg\vphi \notin \mc{I}_\mc{N'}(w')$ for every $w'\geq f(w)$, in order to meet the requirement $\mrm{Ad}_1 (\mc{N}, \mc{N}', f)$. Hence, $f(w)\Vdash_{\mc{N}'} \neg \T\corn{\neg \vphi}$, and the claim follows. The claim in brackets is similar.

ii) and iii) are similar. 
\end{proof}

\noindent Finally, we can easily establish the relation between the four schemes:

\begin{prop}
Let $X\subseteq \nat$. Then $\mc{J}_\mrm{svi}(X)\subseteq \mc{J}_\mrm{vb}(X)\subseteq \mc{J}_\mrm{vc}(X)\subseteq \mc{J}_\mrm{mc}(X)$.
\end{prop}
\begin{proof}
Follows by the following fact: if $\mc{M}$ is a structure and $w$ a world in $W_\mc{M}$, then for all structures $\mc{M}'$:

$\mrm{Ad}_3 (\mc{M}, \mc{M}', f)$ implies $\mrm{Ad}_2 (\mc{M}, \mc{M}', f)$

$\mrm{Ad}_2 (\mc{M}, \mc{M}', f)$ implies $\mrm{Ad}_1 (\mc{M}, \mc{M}', f)$

$\mrm{Ad}_1 (\mc{M}, \mc{M}', f)$ implies $\mc{I}_\mc{M}\leqq_f \mc{I}_{\mc{M}'} $
\end{proof}

\section{The axiomatic theory}\label{axiomatics}

In section \ref{semantics}, we introduced fixed-point structures for the scheme SVI. Clearly, this definition is easily applicable to the cases of VBI, VCI, and MCI. Now, the peculiarity of these fixed-point structures is that, just like their classical counterpart, they model supervaluational \textit{axiomatic} theories of truth.

\begin{dfn}\label{ISV}
$\mrm{ISV}$ is the theory extending $\mrm{HAT}$ with the following axioms: 

\begin{description}

\item[$(\mrm{ISV}1)$] $\forall s,t((\T(s\subdot{=}t) \leftrightarrow s^{\circ} = t^{\circ})\wedge (\T(s\subdot{\neq}t) \leftrightarrow s^{\circ} \neq t^{\circ}))$

\item[$(\mrm{ISV}2)$] $\forall x(\mrm{Ax}_{\mrm{HAT}}(x)\rightarrow \T x)$

\item[$(\mrm{ISV}3)$] $\forall x,v (\mrm{Sent}_{\lt}(\forall v x)\ra (\forall t \T x(t/v))\ra \T (\subdot \forall v x))$

\item[$(\mrm{ISV}4)$] $\forall t(\T t^\circ \ra \T\subdot\T t)$

\item[$(\mrm{ISV}5)$] $\forall x,y (\mrm{Sent}_{\lt}(x\subdot\ra y)\ra(\T ( x\subdot \ra y)\rightarrow (\T x\ra \T y )))$

\item[$(\mrm{ISV}6)$] $\forall x(\T\ulcorner \T(\dot{x}) \to \mathrm{Sent}(\dot{x})\urcorner)$

\item[$(\mrm{ISV}7)$] $\T\ulcorner \vphi\urcorner \rightarrow \vphi$ for any formula $\vphi$ of $\lt$. 

\end{description} 
\end{dfn}

\begin{thm}\label{soundness of ISV}
$\mrm{ISV}$ is sound with respect to fixed-point structures of $\mrm{svi}$, i.e., if $\mc{M}$ is a fixed-point structure of $\mrm{svi}$, then $\mc{M}\vDash \mrm{ISV}$. 
\end{thm}

\begin{proof}

Take $\mc{M}$ to be a fixed-point structure of $\mrm{svi}$. Let $w$ be a world in $\mc{M}$, and let $\mc{I}_\mc{M}(w)=X$.

\begin{itemize}
    \item[ISV1] For the first conjunct, assume $s, t$ are closed terms and $\T(s\subdot =t)$ holds. Then, $s\subdot =t\in X$. Now pick a structure $\mc{N}$ and world $v$ such that $\mc{I}_\mc{N}(v)=X$, and we have $v\Vdash_\mrm{svi}^{\mc{N}} s^\circ = t^\circ$. So for any structure $\mc{N}'$ whose interpretation embeds $\mc{I}_\mc{N}$ via a function $f$, $f(v)\Vdash_{\mc{N}'} s^\circ = t^\circ$. Since all structures agree on arithmetical formulae and are $\Delta_0$-sound, $s^\circ = t^\circ$ is a true atomic arithmetical formula. The reverse reasoning yields the other direction. The other conjunct is obtained similarly.
    \item[ISV2] If $\vphi$ is an axiom of HAT, it holds in all worlds. Therefore, given any structure $\mc{N}$ and world $v$ with $I_\mc{N}=X$, and given any EI function $f$ into an interpretation of a structure $\mc{N}'$, the axiom holds in $f(v)$.
    \item[ISV3] Suppose $x$ is the numeral of the code of $\vphi$. The antecedent of ISV3 implies $\#\vphi(\bar n)\in X$ for all $n\in\omega$. Then: 
    
    i) Taking any structure $\mc{N}$ and world $v$ with $I_\mc{N}=X$, and given any EI function $f$ from $I_\mc{N}=X$ to an interpretation of a structure $\mc{N}'$, $f(v)\Vdash_{\mc{N}'}\vphi(\bar n)$ for all $n\in \omega$. 

    ii) By the hereditariness condition, $\vphi(\bar n)$ for all $n\in \omega$ is forced at any world accessible from $f(v)$. 

    iii) Therefore, $f(v)\Vdash_{\mc{N}'}\forall v\vphi$, whence $v\Vdash^{\mc{N}}_\mrm{svi} \forall v\vphi$; and thus, since $v, \mc{N}$ were arbitrary, $\#\forall v\vphi\in X$.
    \item[ISV4] Follows from the fixed-point nature of $X$.
    
    \item[ISV5] Let $m= \# \vphi$, $n= \# \psi$, and assume $m, m\subdot \rightarrow n\in X$. Take any structure $\mc{N}$ and world $v$ with $I_\mc{N}(v)=X$; consider any EI function $f$ from $I_\mc{N}$ to an interpretation of a structure $\mc{N}'$. Since $\mc{I}_{\mc{N}}\leqq_f \mc{I}_{\mc{N}'}$, $f(v)\Vdash_{\mc{N}'} \vphi \wedge (\vphi\rightarrow \psi) $, whence $f(v)\Vdash_{\mc{N}'} \psi$ follows. This establishes $v\Vdash^{\mc{N}}_\mrm{svi}\psi$, and hence $n\in X$.

    \item[ISV6] Follows from the fact that only sentences are forced at worlds. 

    \item[ISV7] The claim is established from the following other claim: 

    \begin{equation*}\label{classical soundness}
    w\Vdash^{\mc{M}}_\mrm{svi}\vphi \Rightarrow w\Vdash_{\mc{M}}\vphi \tag{$\star$}
    \end{equation*}

    This claim is straightforward from the fact that $\mc{I}_\mc{M}\leqq_\mrm{id} \mc{I}_\mc{M}$, where $\mrm{id}$ is the identity function. The main claim now follows: since $\#\vphi\in X=\mc{J}_\mrm{svi}(X)$, and $w,\mc{M}$ are such that $\mc{I}_\mc{M}(w)=X$, we obtain $w\Vdash^{\mc{M}}_\mrm{svi}\vphi $, whence $ w\Vdash_{\mc{M}}\vphi$ by \ref{classical soundness}. 
\end{itemize}
\end{proof}

Here are some facts and theorems of ISV, which essentially mimic those proven by Cantini's theory VF:

\begin{prop}
\begin{enumerate}
    \item[i)]  $\mrm{ISV}\vdash \forall x (\mrm{Pr}_{\mrm{HAT}}(x)\ra \T x)$ 
    \item[ii)] $\mrm{ISV}\vdash \T \corn{\vphi}\vee \T \corn{\neg \vphi}  \ra ((\vphi\ra \T\corn{\psi})\leftrightarrow \T\corn{\vphi\ra \psi})$, for any formulae $\vphi, \psi$ of $\lt$.
    \item[iii)] $\mrm{ISV}\vdash \forall x, y(\mrm{Sent}_{\lt}(x\subdot\wedge y)\ra(\T (x \subdot \wedge y)\leftrightarrow \T x \wedge \T y))$ 
    \item[iv)] $\mrm{ISV}\vdash \forall x,y(\mrm{Sent}_{\lt}(x\subdot \vee y)\ra (\T x \vee \T y \ra \T (x\subdot \vee y)))$
    \item[v)] $\mrm{ISV}\vdash \forall x,v (\mrm{Sent}_{\lt}(\forall v x)\ra (\forall z \T (x (z/v))\leftrightarrow \T(\subdot \forall v x)))$ 
    \item[vi)]  $\mrm{ISV}\vdash \forall x,v (\mrm{Sent}_{\lt}(\exists v x)\ra(\exists z \T (x (z/v))\ra \T(\subdot \exists v x)))$
    \item[vii)] $\mrm{ISV}\vdash \forall x (\mrm{Sent}_{\lt}(x)\ra(\T x\ra  \T (\subdot \neg \subdot \neg x))  $
    \item[viii)] $\mrm{ISV}\vdash  \T \corn{\vphi}\leftrightarrow \T(\subdot \T \corn{\vphi})$ for every sentence $\vphi$ of $\lt$.
    \item[ix)] If $\vphi$ is a sentence of $\lnat$, then $\mrm{ISV}\vdash  \T \corn{\vphi}\leftrightarrow \vphi$
    \item[x)]  $\mrm{ISV}\vdash  \forall x(\T x \wedge \T \subdot \neg x \ra \bot)$
\end{enumerate}
\end{prop}

\begin{proof}
The proof of items i-ix is exactly like that of \cite[Prop. 2.1]{cantini_1990}. Note that ii) goes through because the inference $\neg p \vee q \vdash p\ra q$ is valid in IL. Also, note that in vii), the other direction does not hold because the inference $\neg \neg p\vdash p$ is not valid in IL. 

For x): follows directly from ISV1, ISV2, ISV5. 
\end{proof}

\noindent We can also present axiomatizations for the fixed-points of the schemes vci and mci. 

\begin{dfn}
    The theory $\mrm{IVB}$ consists of axioms ISV1-ISV7 plus the axiom:

  \begin{enumerate}
        \item[$\mrm{(IVB8)}$] $\mrm{ISV}\vdash \forall x(\T\subdot \neg x \ra \T\subdot \neg \subdot \T x)$
    \end{enumerate}
    
    \noindent The theory $\mrm{IVF}$ consists of axioms ISV1-ISV7 and IVB8 plus the axiom: 

    \begin{enumerate}
        \item[$\mrm{(IVF9)}$] $\mrm{ISV}\vdash \forall x(\T\corn{\T \dot x \wedge \T \subdot \neg \dot x \ra \bot})$
    \end{enumerate}

   \noindent The theory $\mrm{IMC}$ consists of axioms ISV1-ISV6 and IVB8 plus the axiom: 

    \begin{enumerate}
        \item[$\mrm{(IMC8)}$] $\mrm{ISV}\vdash \forall x(\T\corn{\T (\subdot \neg \dot x) \lra (\T \dot x\ra \bot)})$
    \end{enumerate}

\end{dfn}

\noindent The theory IVF is the intuitionistic counterpart of Cantini's theory VF from \cite{cantini_1990}.Cantini develops that theory as a sound axiomatization of the scheme vc over classical logic. For their part, the theories IVB, IVF and IMC are sound with respect to fixed-points of VBI, VCI and MCI, respectively. The classical counterparts of IVB and IMC are the theories $\mrm{VF}^-$ and $\mrm{VFM}$, introduced in \cite{dopico_hayashi_2024}.

\begin{prop}
\begin{enumerate}
    \item[i)] $\mrm{IVB}$ is sound with respect to fixed-point structures of $\mrm{VBI}$, i.e., if $\mc{M}$ is a fixed-point structure of $\mrm{VBI}$, then $\mc{M}\vDash \mrm{IVB}$.
    \item[ii)] $\mrm{IVF}$ is sound with respect to fixed-point structures of $\mrm{VCI}$, i.e., if $\mc{M}$ is a fixed-point structure of $\mrm{VCI}$, then $\mc{M}\vDash \mrm{IVF}$.
    \item[iii)] $\mrm{IMC}$ is sound with respect to fixed-point structures of $\mrm{MCI}$, i.e., if $\mc{M}$ is a fixed-point structure of $\mrm{MCI}$, then $\mc{M}\vDash \mrm{IMC}$.
\end{enumerate}
\end{prop}

\begin{proof}
For i): the only difference with ISV is axiom IVF8, which follows from Proposition \ref{alternative schemes}.i.

For ii): the only difference with ISV is IVF9, which follows from Proposition \ref{alternative schemes}.ii.

For iii): same with Proposition \ref{alternative schemes}.iii. 
\end{proof}

We close this section by noting the lower-bound of the proof-theoretic strength of ISV. The result is due to Leigh in \cite{leigh_2013}, who shows that a theory very much in the style of ISV (in fact, a super-theory thereof) is proof-theoretically equivalent to $\mrm{ID}^i_1$, the intuitionistic version of the theory of one inductive definition. For details on the latter, we refer the reader to \cite[\S 4]{leigh_2013}. Thus, by inspecting that proof, we see that all resources are available in ISV, and therefore:

\begin{prop}\label{proof theory ISV}
$\vert \mrm{ISV}\vert \equiv \vert \mrm{ID}^i_1\vert$
\end{prop}

\section{The fixed-frame approach}\label{fixed-frame approach}

As repeatedly stated, our supervaluational-forcing relation $w\Vdash^\mc{M}_\mrm{svi}\vphi$ ranges over structures and embedding functions that might contain more worlds and enlarged accessibility relations than the original Kripke model ($\mc{M}$) for whose world ($w$) the relation is defined. Likewise, the Kripke jump for a set $X\subseteq \mrm{SENT}_{\lt}$ ranges over all structures which possess a world with truth interpretation $X$, offering a global sense of what supervaluationally follows, in intuitionistic logic, from a given set of sentences. This strategy, however, is to be contrasted with the approach that some papers have adopted when building fixed-point constructions over possible-world semantics---e.g, \cite{halbach_leitgeb_welch_2003}, \cite{halbach_welch_2009}, \cite{nicolai_2018a}, \cite{stern_2025}. Although, with the exception of Stern, the rest of authors are primarily concerned with a fixed-point construction for a necessity predicate, the framework applies equally well for truth. The goal of this section is to characterize this approach, that we shall call the \textit{fixed-frame} approach, and compare it to the one developed in section \ref{semantics}.

Consider a persistent first-order intuitionistic Kripke $\omega$-structure $\mc{M}:=\langle W_\mc{M}, \leq_\mc{M}, \mc{D}_\mc{M}, \mc{I}_\mc{M}\rangle$. The fixed-frame approach, as its name makes clear, ``fixes'' the frame (i.e., the worlds and the accessibility relation), and considers only interpretation extensions that keep the same set of worlds as their domain. That is, the extensions of $\mc{I}_\mc{M}$ considered are interpretation relations $\mc{I}'_\mc{M}\subseteq W_\mc{M}\times \mrm{SENT}_{\lt}$ with $\mc{I}_\mc{M}\subseteq \mc{I}'_\mc{M}$. No additional worlds, and hence no enlarged accessibility relation, is allowed. There is also no reference to embeddings, so interpretations for isomorphic structures with different labels for the worlds are left aside. Let us write $\mc{I}_\mc{M}\preceq \mc{I}'_\mc{M}$ when $\mc{I}'_\mc{M}$ is an extension of $\mc{I}_\mc{M}$ thus understood. Since we are keeping the frame the same, let us write $w, \mc{I}_\mc{M}\Vdash\vphi$ or $w, \mc{I}'_\mc{M}\Vdash\vphi$ to indicate that the frame is fixed, and the interpretation relation in question is $\mc{I}_\mc{M}$, or $\mc{I}'_\mc{M}$, respectively---instead of providing a new structure for each new interpretation. We maintain this notation for the rest of the section, even when the frame varies. For the world $w\in W_\mc{M}$, $w$ and $\mc{I}_\mc{M}$ \textit{ff-forcing} the $\lt$-sentence $\vphi$ is then defined as
\begin{equation*}
    w, \mc{I}_\mc{M}\Vdash_\mrm{ff}\vphi :\Lra \forall\, \mc{I}'_\mc{M}\succeq \mc{I}_\mc{M}(w, \mc{I}'_\mc{M}\Vdash \vphi)
\end{equation*}
\noindent Moreover, the Kripke jump is defined on interpretation relations, not on sets of sentences.\footnote{To be wholly faithful to the original approaches, authors tend to conceive of the interpretation for truth $\mc{I}_\mc{M}$ for a model $\mc{M}$ as a \textit{function} from $W_\mc{M}$ to $\mc{P}(\mrm{SENT}_{\lt})$, rather than a subset of $W_\mc{M}\times \mrm{SENT}_{\lt}$.} The idea is that, when applied to an evaluation function $\mc{I}_\mc{M}:W_\mc{M}\times \mrm{SENT}_{\lt}$, the jump $\mc{J}^\mc{M}_\mrm{ff}$ yields another evaluation function $\mc{I}'_\mc{M}:W_\mc{M}\times \mrm{SENT}_{\lt}$ such that, for each $w\in W_\mc{M}$: 
\begin{equation*}
    \mc{I}'_\mc{M} (w) = \mc{J}^\mc{M}_\mrm{ff}(\mc{I}_\mc{M}) (w) = \{ \#\vphi \sth w, \mc{I}_\mc{M}\Vdash_\mrm{ff}\vphi\}
\end{equation*}

\noindent This jump is monotonic, and hence will yield fixed-points, including a minimal one (see, e.g., \cite[\S 4.2, \S 5]{nicolai_2018a}). 

So how does the approach outlined in the previous sections compare to the fixed-frame approach? Two features of the fixed-frame approach stand out. On the one hand, as mentioned, this approach is tied to a particular structure, with a fixed frame. Additionally, the fixed-frame approach is based on interpretation relations: it yields, as fixed-points, an interpretation relation for the frame of the model $\mc{M}$ over which the construction is carried out. By contrast, our fixed-points are not tied to a particular structure, but rather, have been obtained by supervaluating over every possible structure. This picture is, arguably, more faithful to the spirit of classical supervaluationism, where one really aims to range over \textit{all} models possessing some features. Moreover, the construction renders sets of sentences as fixed-points: this is, on the one hand, more in line with the classical case; and, on the other, seems advantageous when it comes to discerning `what is true' intuitionistically, irrespective of the underlying structure. On top of that, but related to this last point, the sentence-based jump (as opposed to the interpretation-based fixed-frame jump) might allow for a better and more homogeneous philosophical reading of the truth construction propounded by Kripke. For instance, it is often claimed that the minimal fixed-point of Kripke's constructions enjoys a privileged position.\footnote{See \cite{kremer_1988} for an in-depth discussion. Note that this is not a claim made by Kripke in \cite{kripke_1975}.} Our approach can preserve this idea: in starting with the empty set, $\varnothing$, one obtains a minimal fixed-point of $\mc{J}_\mrm{svi}$ standing for the least set of sentences that is supervaluationally true under intuitionistic logic. The fixed-frame approach, on the contrary, has no `overall' least fixed-point. Instead, it requires that a frame be given in advance---only then can the least fixed-point be obtained, which is nonetheless only a least fixed-point relative to the frame. 

At the same time, the fixed-frame approach enjoys a great advantage with respect to the construction we have introduced: transparent models for the truth predicate can be given for any frame whatsoever. Indeed, it can be checked that if $\mc{M}:=\langle W_\mc{M}, \leq_\mc{M}, \mc{D}_\mc{M}, \mc{I}_\mc{M}\rangle $ is a persistent first-order intuitionistic Kripke $\omega$-structure, with $\mc{J}^\mc{M}_\mrm{ff}(\mc{I}_\mc{M})=\mc{I}_\mc{M}$, 
\begin{equation*}
    w, \mc{I}_\mc{M}\Vdash_\mrm{ff}\vphi \Lra w, \mc{I}_\mc{M}\Vdash_\mrm{ff}\T\corn\vphi
\end{equation*}
\noindent holds for all $w\in W_\mc{M}$ and formulae $\vphi$.

Playing around with the admissibility conditions of the SVI scheme, however, we can establish connections between fixed-points of our approach and fixed-points of the fixed-frame approach. For instance, fix a frame $\langle W_\mc{M}, \leq_\mc{M}, \mc{D}_\mc{M}\rangle$. Now, take a persistent first-order intuitionistic Kripke $\omega$-structure $\mc{N}:=\langle W_\mc{N}, \leq_\mc{N}, \mc{D}_\mc{N}, \mc{I}_\mc{N}\rangle$, and define:
\begin{multline*}
    w, \mc{I}_\mc{N}\Vdash^\mc{N}_{\mrm{svi}_\mc{M}}\vphi:\Lra \forall \mc{N}', f(\mc{I}_\mc{N}\leqq_f\mc{I}_{\mc{N}'} \& \leq_\mc{N'}\cong_\tau \leq_\mc{M} \& \forall u\in W_\mc{N'} (\mc{I}_\mc{N'}(u)\supseteq \mc{I}_\mc{M}(\tau(u)))\\
    \Ra f(w), \mc{I}_\mc{N'}\Vdash_\mc{N'} \vphi)
\end{multline*}
\sloppy Here, $\leq_\mc{N'}\cong_\tau \leq_\mc{M}$ indicates that $\tau$ is an isomorphism between the structure $(W_\mc{N}', \leq_\mc{N'})$ and $(W_\mc{M}, \leq_\mc{M})$. Despite the excessive symbolism of the formula, the meaning is immediate: in supervaluationally forcing over $w$, we only consider interpretation extensions which are isomorphic to $\mc{M}$ and whose interpretations of truth at every world include at least all the sentences assigned to the corresponding world in $\mc{M}$. In doing so, we are `freezing' the frame (up to isomorphism) in a way similar to what the fixed-frame approach does. It is easy to see that, when $\mc{N}$ is $\mc{M}$ itself, svi$_\mc{M}$-forcing is equivalent to ff-forcing:
\begin{prop}\label{equivalence forcing same structure}
Let $\mc{M}:=\langle W_\mc{M}, \leq_\mc{M}, \mc{D}_\mc{M}, \mc{I}_\mc{M}\rangle$ be a persistent, first-order intuitionistic Kripke $\omega$-structure. Then:
\begin{equation*}
    w, \mc{I}_\mc{M}\Vdash^\mc{M}_{\mrm{svi}_\mc{M}}\vphi\Lra w, \mc{I}_\mc{M}\Vdash_\mrm{ff}\vphi
\end{equation*}
\end{prop}
\begin{proof}
The left-to-right direction is immediate, since any ff-extension $\mc{I}'_\mc{M}$ meets the admissibility condition of $\Vdash^\mc{M}_{\mrm{svi}_\mc{M}}$. The right-to-left direction is also straightforward: if there is a structure $\mc{N}$ that is isomorphic to $\mc{M}$ and which serves as a counterexample to $w, \mc{I}_\mc{M}\Vdash^\mc{M}_{\mrm{svi}_\mc{M}}\vphi$, it suffices to re-label the worlds in $\mc{N}$ according to the bijective function so that these match the worlds in $\mc{M}$, and this will be a counterexample to $w, \mc{I}_\mc{M}\Vdash_\mrm{ff}\vphi$.
\end{proof}
\noindent We then write:
\begin{equation*}
    \mc{J}^\mc{M}_\mrm{svi}(X):=\{\# \vphi \sth \forall \mc{M}', \forall w\in W_\mc{M'}(\mc{I}_{\mc{M}'} (w)=X \Rightarrow w, \mc{I}_\mc{M'}\Vdash^{\mc{M}'}_{\mrm{svi}_\mc{M}}\vphi)\}
\end{equation*}
\noindent It can now be shown that fixed-points of $\mc{J}^\mc{M}_\mrm{ff}$ generate a fixed-point for $\mc{J}^\mc{M}_\mrm{svi}$. More specifically, and regardless of what the frame that $\mc{M}$ is based on is like, the intersection of all the interpretations of truth for the worlds in $\mc{M}$ is a fixed-point of $\mc{J}^\mc{M}_\mrm{svi}$. 

\begin{thm}
Let $\mc{M}:=\langle W_\mc{M}, \leq_\mc{M}, \mc{D}_\mc{M}, \mc{I}_\mc{M}\rangle$ be a persistent, first-order intuitionistic Kripke $\omega$-structure. Let $\mc{J}^\mc{M}_\mrm{ff}(\mc{I}_\mc{M})=\mc{I}_\mc{M}$. Then:
\begin{equation*}
    \bigcap_{w\in W_\mc{M}}\mc{I}_\mc{M}(w)=\mc{J}^\mc{M}_\mrm{svi}(\bigcap_{w\in W_\mc{M}}\mc{I}_\mc{M})
\end{equation*}
\end{thm}
\begin{proof}
For the sake of brevity, we will let $X$ stand for $\bigcap_{w\in W_\mc{M}}\mc{I}_\mc{M}(w)$.\\
\noindent Left-to-right: let $\#\vphi\in X$. Then, for any $w\in W_\mc{M}$,
\begin{equation}
    \forall \mc{I}'_\mc{M}\succeq \mc{I}_\mc{M}(w, \mc{I}'_\mc{M}\Vdash\vphi)\label{equan1}
\end{equation}
In order to show $\#\vphi\in \mc{J}^\mc{M}_\mrm{svi}(X)$, take a structure $\mc{N}=\langle W_\mc{N}, \leq_\mc{N}, \mc{D}_\mc{N}, \mc{I}_\mc{N}\rangle$ and a world $n\in W_\mc{N}$ s.t. $\mc{I}_\mc{N}(n)=X$. Since we need to show $n, \mc{I}_\mc{N}\Vdash^\mc{N}_{\mrm{svi}_\mc{M}}\vphi$, take further a structure $\mc{P}=\langle W_\mc{P}, \leq_\mc{P}, \mc{D}_\mc{P}, \mc{I}_\mc{P}\rangle$ with 
\begin{equation*}
    \mc{I}_\mc{N}\leqq_f\mc{I}_{\mc{P}} \text{ and } \leq_\mc{P}\cong_\tau \leq_\mc{M} \text{ and } \forall p\in W_\mc{P} (\mc{I}_\mc{P}(p)\supseteq \mc{I}_\mc{M}(\tau(p)))
\end{equation*}
If no such structure exists, $n, \mc{I}_\mc{N}\Vdash^\mc{N}_{\mrm{svi}_\mc{M}}\vphi$ follows trivially. If it does exist, it follows that there is a structure $\mc{P}'$ with 
\begin{equation*}
    W_{\mc{P}'}=W_\mc{M} \text{ and } \leq_\mc{P'} = \leq_\mc{M} \text{ and }\leq_\mc{P}\cong_\pi \leq_\mc{P'} \text{ and } \forall p\in W_\mc{P'} (\mc{I}_\mc{P'}(p) = \mc{I}_\mc{P}(\pi(p)))
\end{equation*}
It is immediate to check then that, for all $p\in W_\mc{P'}$
\begin{equation}
    p, \mc{I}_\mc{P'}\Vdash_{\mc{P}'}\vphi \Lra \pi(p), \mc{I}_\mc{P}\Vdash_{\mc{P}}\vphi \label{equan2}
\end{equation}
Take now $f(n)$, i.e., the image of $n$ under the EI function from $\mc{N}$ into $\mc{P}$. In order to prove $n, \mc{I}_\mc{N}\Vdash^\mc{N}_{\mrm{svi}_\mc{M}}\vphi$, we need to show $f(n), \mc{I}_\mc{P}\Vdash\mc{P}\vphi$. From (\ref{equan2}), this obtains iff $\tau^{-1}(f(n)), \mc{I}_\mc{P'}\Vdash_\mc{P'} \vphi$. Since it is easy to see that $\mc{I}_\mc{M}\preceq \mc{I}_\mc{P'}$, our assumption (\ref{equan1}) implies $\tau^{-1}(f(n)), \mc{I}_\mc{P'}\Vdash \vphi$.

\noindent Right-to-left: assume $\#\vphi\in \mc{J}^\mc{M}_\mrm{svi}(X)$, i.e., 
\begin{equation}
    \forall \mc{N}, \forall n\in W_\mc{N}(\mc{I}_\mc{N}(n)=X \Ra n,\mc{I}_\mc{N} \Vdash_{\mrm{svi}_\mc{M}}^\mc{N}\vphi) \label{equan3}
\end{equation}
We now need to show that, for any $w\in W_\mc{M}$,
\begin{equation}
    \forall \mc{I}'_\mc{M}\succeq \mc{I}_\mc{M}(w, \mc{I}'_\mc{M}\Vdash\vphi)
\end{equation}
So take $w'\in W_\mc{M}$. Take $\mc{I}'_\mc{M}\succeq \mc{I}_\mc{M}$. From (\ref{equan3}), we can instantiate the claim to a particular $\mc{N}$, given by the same worlds and accessibility relation as $\mc{M}$, and with $\mc{I}_\mc{N}(w_0)=X$. Then, $w_0, \mc{I}_\mc{N}\Vdash_{\mrm{svi}_\mc{M}}^\mc{N}\vphi$. The latter implies that, for all structures $\mc{P}=\langle W_\mc{P}, \leq_\mc{P}, \mc{D}_\mc{P}, \mc{I}_\mc{P}\rangle$ with 
\begin{equation*}
    \mc{I}_\mc{N}\leqq_f\mc{I}_{\mc{P}} \text{ and } \leq_\mc{P}\cong_\tau \leq_\mc{M} \text{ and } \forall p\in W_\mc{P} (\mc{I}_\mc{P}(p)\supseteq \mc{I}_\mc{M}(\tau(p)))
\end{equation*}
we have 
\begin{equation}
f(w_0), \mc{I}_\mc{P}\Vdash \vphi \label{equan4}  
\end{equation}
If $\mc{I}'_\mc{M}\succeq \mc{I}_\mc{M}$, it is straightforward to see that $\mc{I}_\mc{N}\leqq_\mrm{id} \mc{I}'_\mc{M}$, where $\mrm{id}$ is the identity function. Trivially, $\leq_\mc{M}\cong_\tau \leq_\mc{M}$, with $\tau=\mrm{id}$; and, moreover, $\forall w'\in W_\mc{M} (\mc{I}'_\mc{M}(w')\supseteq \mc{I}_\mc{M}(\tau(p)))$. Therefore, by (\ref{equan4}), $w_0, \mc{I}'_\mc{M}\Vdash\vphi$.
\end{proof}

\noindent Since the intersection of the interpretation of truth for a one-world frame is just the interpretation of truth for the single world, we have the following corollary: 
\begin{corollary}
Let $\mc{M}:=\langle W_\mc{M}, \leq_\mc{M}, \mc{D}_\mc{M}, \mc{I}_\mc{M}\rangle$ be a persistent, first-order intuitionistic Kripke $\omega$-structure with $\vert W_\mc{M}\vert =1$, and let $w\in W_\mc{M}$. Let $\mc{J}^\mc{M}_\mrm{ff}(\mc{I}_\mc{M})=\mc{I}_\mc{M}$. Then:
\begin{equation*}
    \mc{I}_\mc{M}(w)=\mc{J}^\mc{M}_\mrm{svi}(\mc{I}_\mc{M}(w))
\end{equation*}
\end{corollary}

\noindent The converse direction of the above result also holds, i.e., from a given fixed-point of $\mc{J}_\mrm{svi}$ we can obtain a fixed-point of $\mc{J}_\mrm{ff}$.

\begin{prop}
Let $\mc{M}:=\langle W_\mc{M}, \leq_\mc{M}, \mc{D}_\mc{M}, \mc{I}_\mc{M}\rangle$ be a persistent, first-order intuitionistic Kripke $\omega$-structure with $\vert W_\mc{M}\vert =1$, and let $w\in W_\mc{M}$. Let $\mc{I}_\mc{M}(w)=X=\mc{J}_\mrm{svi}^\mc{M}(X)$. Then, $\mc{I}_\mc{M}=\mc{J}_\mrm{ff}^\mc{M}(\mc{I}_\mc{M})$.
\end{prop}
\begin{proof}
Left-to-right: assume $\#\vphi \in \mc{I}_\mc{M}(w)$. We want to show $\forall \mc{I}'_\mc{M}\succeq \mc{I}_\mc{M}(w, \mc{I}'_\mc{M}\Vdash \vphi)$. Since $\#\vphi \in \mc{I}_\mc{M}(w)=\mc{J}_\mrm{svi}^\mc{M}(X)$, 
\begin{equation*}
    \forall \mc{N}, \forall n\in W_\mc{N}(\mc{I}_\mc{N}(n)=X\Ra n, \mc{I}_\mc{N}\Vdash^\mc{N}_{\mrm{svi}_\mc{M}}\vphi).
\end{equation*}
So instantiate the above with $\mc{N}=\mc{M}$ and $n=w$. Then, $w, \mc{I}_\mc{M}\Vdash^\mc{M}_{\mrm{svi}_\mc{M}}\vphi$. By Proposition \ref{equivalence forcing same structure}, $\forall \mc{I}'_\mc{M}\succeq \mc{I}_\mc{M}(w, \mc{I}'_\mc{M}\Vdash \vphi)$.\\
\noindent Right-to-left: assume $\#\vphi \in \mc{J}^\mc{M}_\mrm{ff}(\mc{I}_\mc{M}(w))$. We want to show that $\#\vphi \in \mc{I}_\mc{M}(w)$. Since the latter is a fixed-point of $\mc{J}^\mc{M}_\mrm{svi}$, it suffices to show:
\begin{equation*}
    \forall \mc{N}, \forall n\in W_\mc{N}(\mc{I}_\mc{N}(n)=X\Ra n, \mc{I}_\mc{N}\Vdash^\mc{N}_{\mrm{svi}_\mc{M}}\vphi).
\end{equation*}
\noindent Fix some $\mc{N}, n\in W_\mc{N}$, and take some structure $\mc{P}$ such that 
\begin{equation*}
    \mc{I}_\mc{N}\leqq_f\mc{I}_{\mc{P}} \text{ and } \leq_\mc{P}\cong_\tau \leq_\mc{M} \text{ and } \forall p\in W_\mc{P} (\mc{I}_\mc{P}(p)\supseteq \mc{I}_\mc{M}(\tau(p)))
\end{equation*}
If no such structure exists, the claim follows trivially. If it does, then, as in the previous proposition, there will be a structure $\mc{P}'$ with 
\begin{equation*}
    W_{\mc{P}'}=W_\mc{M} \text{ and } \leq_\mc{P'} = \leq_\mc{M} \text{ and }\leq_\mc{P}\cong_\pi \leq_\mc{P'} \text{ and } \forall p\in W_\mc{P'} (\mc{I}_\mc{P'}(p) = \mc{I}_\mc{P}(\pi(p)))
\end{equation*} 
So, for all $w\in W_\mc{P'}$ (i.e., only $w$):
\begin{equation*}
    w,\mc{I}_\mc{P'}\Vdash_\mc{P'} \vphi \Lra \tau(w), \mc{I}_\mc{P}\Vdash_\mc{P} \vphi.
\end{equation*}
Then, since $\mc{I}_\mc{P}'\succeq \mc{I}_\mc{M}$, our assumption $\#\vphi \in \mc{J}^\mc{M}_\mrm{ff}(\mc{I}_\mc{M}(w))$ implies that $w, \mc{I}_\mc{P'}\Vdash_\mc{P'} \vphi$. Hence, $\tau(w), \mc{I}_\mc{P}\Vdash_\mc{P} \vphi$. Finally $\tau(w)$ is the only world of $\mc{W}_\mc{P}$, so $\tau(w)=f(n)$, whence $f(n), \mc{I}_\mc{P}\Vdash\vphi$. Therefore, since $\mc{P}$ was arbitrary, $n, \mc{I}_\mc{N}\Vdash^\mc{N}_{\mrm{svi}_\mc{M}}\vphi$.
\end{proof}

\noindent It must be noted that the structures of the fixed-frame approach also yield models of the theory $\mrm{ISV}$. This is in line with existing work on axiomatic theories of truth and necessity (and other modal predicates), as given, e.g., in \cite{stern_2014} or \cite{nicolai_2018a}.  

\begin{prop}
$\mrm{ISV}$ is sound with respect to fixed-point models of $\mc{J}_\mrm{ff}$. That is, if $\mc{M}:=\langle W_\mc{M}, \leq_\mc{M}, \mc{D}_\mc{M}, \mc{I}_\mc{M}\rangle$ is a persistent, first-order intuitionistic Kripke $\omega$-structure with $\mc{J}^\mc{M}_\mrm{ff}(\mc{I}_\mc{M})=\mc{I}_\mc{M}$, then $\mc{M}\vDash \mrm{ISV}$. 
\end{prop}
\begin{proof}
The proof is essentially the same as that of Theorem \ref{soundness of ISV}.
\end{proof}

\noindent Thus, as per their axiomatic consequences, the fixed-frame and our approach seem to us to be on a par.

\section{Compositional supervaluations}\label{section:compositional supervaluations}

One salient feature of intuitionistic logic with respect to classical logic is the greater degree of compositionality it displays at the proof-theoretic level. Thus, Gentzen's classic \citeyearpar{gentzen_1935} already showed that intuitionistic sequent calculus possessed the disjunction property, whereby a disjunction is provable if and only if at least one disjunct is. Likewise, the existence property---i.e., an existentially quantified formula is provable if and only if there is a witness for it---is a well-known feature of even weak arithmetical theories over intuitionistic logic.\footnote{See \cite{friedman_1975}.} The question is whether this can be translated into the theory of truth.

Evidently, compositionality is a very desirable feature of truth theories. Our usual logical connectives are truth functional, and this should intuitively be captured by our theory of truth. Supervaluational theories over classical logic, however, tend to be at odds with compositionality, particularly with the compositionality of disjunctions and existentially quantified statements. Unlike the highly compositional truth theory obtained with the SK schema, the semantic constructions obtained via supervaluations over classical logic fail to deliver compositionality for those logical constants. Moreover, the same occurs with the axiomatic supervaluational theories over classical logic. But, given the connection between intuitionistic logic, and the disjunction and existence properties, one would hope that supervaluational theories of truth over intuitionistic logic might be able to deliver compositionality for said logical constants.

That these hopes are not mere chimeras shows in the fact that Leigh and Rathjen \citeyearpar{leigh_rathjen_2012} already established the consistency of an axiomatic theory---their theory  H$^i$---that looks very much like $\mrm{ISV}$ but includes compositional axioms for disjunction and the existential quantifier. Nonetheless, Leigh and Rathjen provided no supervaluational construction that matched their theory. Hence, the goal of this section is to develop what we call \textit{compositional supervaluations}. These will be the basis of a Kripkean construction that will deliver a supervaluational theory of truth with compositional properties for disjunction and the existential quantifier, along those for conjunction and the universal quantifier. These compositional supervaluations are, however, rather peculiar. Their main peculiarity lies in that this supervaluational forcing relation is determined compositionally for all connectives except for the conditional, making the latter the only one with a clear supervaluational behavior. Nevertheless, as we shall prove, the final theory is fully supervaluational: the fixed-points of the theory can be obtained with a supervaluational jump in the style of section \ref{semantics}.

Another key difference of this approach the forcing relation employed for the underlying models, which is inspired by---but different from---the forcing relation introduced by Do\v sen to deal with non-hereditary Kripke models in \citep{dosen_1991}. This forcing relation takes a ``global" approach to disjunction and the existential quantifier. Such a global approach is what allows for compositionality for these connectives, which are precisely the one for which compositionality tends to be at odds with in supervaluational theories. Let us first introduce the models over which we want our supervaluations to range, together with this new forcing relation. 

\subsection{Global forcing}

In defining the new forcing relation, that we shall call \textit{global} forcing, we want to make sure that it is a satisfactory one. Thus, in this section we will provide a completeness theorem for the global forcing relation with respect to the intuitionistic predicate calculus (although the proof will be given in the appendix). For this purpose, we need to abstract momentarily from the language $\lt$, and consider the landscape of first-order, intuitionistic languages.

Each such first-order, intuitionistic language is denoted by $\mc{L}\subseteq \mc{L}_{\mrm{IPC}}$. Models for these languages are of the form $\mc{M}:=\langle W_\mc{M}, \leq_\mc{M}, \mc{D}_\mc{M}, \mc{I}_\mc{M} \rangle$. $\mc{D}_\mc{M}$ is now allowed to assign any non-empty set to the domain of a world. The union of the range of $\mc{D}_\mc{M}$---in plain words, the domain of the whole structure $\mc{M}$---is denoted by $M$. $\mc{I}_\mc{M}$ is a triple of relations $\mc{M}_\mc{P}$ (assigning an interpretation to predicates), $\mc{M}_\mc{C}$ (assigning an interpretation to constants), and $\mc{M}_\mc{F}$ (assigning an interpretation to functions). Finally, given a closed term $t$, we write $t^\mc{M}$ for the value of the term $t$ under the structure $\mc{M}$ and the world $w\in W_\mc{M}$. Thorough definitions of these points are left for the appendix, so as not to overcrowd this subsection. 

A persistent first-order Kripke model for such an intuitionistic language is one where the interpretation of predicates, constants, and functions for the worlds, as well as the sets assigned to the domain of each world, is non-strictly increasing. This is also properly defined in the appendix. For the sake of simplicity, and as is often done, if $\mc{M}:=\langle W_\mc{M}, \leq_\mc{M}, \mc{D}_\mc{M},  \mc{I}_\mc{M}\rangle$ is a persistent first-order Kripke model for $\mc{L}\subseteq \mc{L}_{\mrm{IPC}}$, we will assume that $\mc{L}$ contains constants $c_d$ for every $d\in M$, and that the object assigned to this constant is, naturally, $d$. In this way, we can dispense with variable assignments---although nothing hangs on this assumption. 

\begin{dfn}\label{global forcing}
Given a persistent first-order Kripke model $\mc{M}:=\langle W_\mc{M}, \leq_\mc{M}, \mc{D}_\mc{M},  \mc{I}_\mc{M}\rangle$ for a  language $\mc{L}\subseteq \mc{L}_{\mrm{IPC}}$, formulae $\vphi$ and $\psi$ of $\mc{L}$, and a world $w\in W_\mc{M}$, we define recursively the relation $w \Vdash_{\mrm{G}} \vphi$ as follows: 

\begin{enumerate}
    \item $w\nVdash_\mrm{G} \bot$.
    \item $w\Vdash_{\mrm{G}}  P_i(t_1,...,t_n)$ iff $\langle P_i, w, \langle t_1,...,t_n\rangle\rangle \in \mc{M}_\mc{P}$, for $P_i(x_1,...,x_n)$ an n-place predicate of $\mc{L}$ and $t_1,...,t_n$ closed terms of $\mc{L}$.
    \item $w\Vdash_{\mrm{G}}  \vphi\wedge \psi$ iff $w\Vdash_{\mrm{G}}  \vphi$ and $w\Vdash_{\mrm{G}}  \psi$.
    \item $w\Vdash_{\mrm{G}}  \vphi\vee \psi$ iff for all $w'\geq w$, $w'\Vdash_{\mrm{G}}  \vphi$ or for all $w'\geq w$, $w'\Vdash_{\mrm{G}}  \psi$.
    \item $w\Vdash_{\mrm{G}}  \vphi\rightarrow \psi$ iff $\forall w'\geq w (w'\Vdash_\mrm{G} \vphi \Ra w'\Vdash_\mrm{G} \psi)$
    \item $w\Vdash_{\mrm{G}}  \exists x \vphi $ iff for all $w'\geq w$ there is some $d\in \mc{D}_\mc{M}(w')$ such that $w\Vdash_{\mrm{G}}  \vphi [x/c_d]$.
    \item $w\Vdash_{\mrm{G}}  \forall x \vphi$ iff for all $w'\geq_\mc{M}w$ and all $d\in \mc{D}_\mc{M}(w')$,  $w'\Vdash_{\mrm{G}}  \vphi [x/c_d]$.
\end{enumerate}

\noindent We write $w\Vdash_{\mrm{G},\mc{M}} \vphi$ when it is deemed appropriate to remark that the underlying structure is $\mc{M}$. Also, $\mc{M}\vDash_\mrm{G} \vphi$ iff $w\Vdash_\mrm{G} \vphi$ for all $w\in W_\mc{M}$. Finally, $\vDash_\mrm{G} \vphi$ iff $\mc{M}\vDash_\mrm{G}\vphi$ holds for any persistent first-order Kripke model $\mc{M}$ of $\mc{L}$.
\end{dfn}

\noindent One can see that the changes with respect to the standard forcing relation concern $\vee$ and $\exists$: in order to make the theory compositional for these two connectives, we give them a `global' character, so that their truth-conditions depend on all worlds accessible from the given one. As stated, in order to show that the G-forcing relation is at least a somewhat meaningful alternative to the standard forcing relation for persistent first-order Kripke structures, we provide a completeness theorem. In particular, we prove: 

\begin{thm}[Completeness of G-models]\label{completeness of G}
Let $\vphi$ be a formula of the language of intuitionistic predicate logic. Then:

\begin{equation*}
    \vdash_\mrm{IPC} \vphi \Leftrightarrow \vDash_\mrm{G} \vphi.
\end{equation*}

\end{thm}

\noindent That is, the intuitionistic predicate calculus is (weakly) complete with respect to the class of persistent first-order Kripke models endowed with the global forcing relation.  We leave the proof to the appendix. We also have the following strong soundness result, whose proof is also given in the appendix:

\begin{corollary}[Strong soundness of G]\label{strong soundness G 1}
Let $\vphi$ be a formula of some first-order language $\mc{L}\subseteq \mc{L}_\mrm{IPC}$, let $\Gamma$ be a set of formulae of $\mc{L}$, and let $\mc
{M}$ range over first-order Kripke models. Then: 

\begin{equation*}
   \Gamma  \vdash_{\mrm{IPC}} \vphi \Ra \, \forall \mc{M}(\mc{M}\vDash_\mrm{G} \Gamma \Ra \mc{M}\vDash_\mrm{G} \vphi)
\end{equation*}
\end{corollary}

\noindent The right-to-left direction of Corollary \ref{strong soundness G 1}, what is often called strong completeness, is not immediate. We conjecture that it is provable, but we could not establish it. Our proof of Theorem \ref{completeness of G} is based on a similar result that Do\v sen proved for his forcing relation in  \cite{dosen_1991}, which builds on a translation with the modal logic S4. In fact, Do\v sen established a strong completeness result. However, Do\v sen's result only covered the case of intuitionistic propositional logic. Extending it to the predicate case becomes difficult, as the translation into the modal language is no longer suitable. Similarly, we could have established a strong completeness result for the propositional fragment of intuitionistic logic with respect to the G-forcing relation, but fail to deliver the same for the predicate case. 

\subsection{From global forcing to compositional supervaluations}

We can now work out the compositional supervaluations on the basis of the forcing relation $\Vdash_\mrm{G}$. As usual, this requires that we first provide a semantic scheme. Yet this semantic scheme, unlike in all previous cases, is defined compositionally, i.e., piecewise for the different logical connectives. At first, then, the scheme will only seem to possess a supervaluational flavour for one connective: the conditional. So, while we promised that the semantic theory thus generated would be essentially supervaluational, we will have to wait until we can show why. 

\sloppy In what follows, $\mc{M}, \mathcal{N}, \mc{M}', $ etc. range over \textit{persistent first-order intuitionistic Kripke $\omega$-structures}, as defined in Definition \ref{persistent FO int structures}, and $w, w', w_i,$ etc. range over worlds in these structures. We keep the same notions of embedding-interpretation function and interpretation embedding as given, respectively, in definitions \ref{embedding function} and \ref{interpretation embedding}. The notion of truncated structure, which was given in Definition \ref{truncated structure}, will be henceforth very much employed. 

\begin{remark}
Since we are working with persistent $\omega$-structures, the G-forcing global clause for the universal quantifier simply amounts to: 
\begin{equation*}
    w\Vdash_{\mrm{G},\mc{M}}\forall v\psi :\Leftrightarrow  \forall n\in\omega (w\Vdash_{\mrm{G},\mc{M}} \psi(\bar n))
\end{equation*}
\noindent For ease of presentation, we will consider this clause, rather than the full clause of Definition \ref{global forcing}, in the rest of this section.
\end{remark}

\begin{dfn}
Given a persistent first-order intuitionistic Kripke $\omega$-structure $\mc{M}:=\langle W_\mc{M}, \leq_\mc{M}, \mc{D}_\mc{M}, \mc{I}_\mc{M} \rangle$, a world $w\in W_\mc{M}$, and a formula $\vphi$ of $\lt$, we define the relation $w\Vdash^{\mc{M}}_\mrm{csv}\vphi$ inductively as follows:

\begin{enumerate}
    \item $w\Vdash^{\mc{M}}_\mrm{csv}R(t_1,...,t_n)  :\Leftrightarrow w\Vdash_{\mrm{G},\mc{M}} R(t_1,...,t_n)$ for $R(x_1,...,x_n)$ an atomic sentence of $\lt$, $t_1,...,t_n$ closed terms.\footnote{This clause can also be given as: 
    \begin{center}
        $w\Vdash^{\mc{M}}_\mrm{csv}R(t_1,...,t_n)  :\Leftrightarrow \forall \mc{M}', f(\mc{I}_\mc{M}\leqq_f \mc{I}_{\mc{M}'} \, \Rightarrow f(w)\Vdash_{\mrm{G},\mc{M}'} R(t_1,...,t_n))$ for $R(x_1,...,x_n)$ an atomic sentence of $\lt$, $t_1,...,t_n$ closed terms.
    \end{center}
    While this gives the clause a more supervaluational flavour, it is, in the end, superfluous.}
    \item $w\Vdash^{\mc{M}}_\mrm{csv}\psi\vee \chi  :\Leftrightarrow \forall w'\geq_\mc{M} w(w'\Vdash^{\mc{M}}_\mrm{csv}\psi)$ or $\forall w'\geq_\mc{M} w(w'\Vdash^{\mc{M}}_\mrm{csv}\chi)$
     \item $w\Vdash^{\mc{M}}_\mrm{csv}\psi\wedge \chi  :\Leftrightarrow w\Vdash^{\mc{M}}_\mrm{csv}\psi$ and $w\Vdash^{\mc{M}}_\mrm{csv}\chi$
     \item $w\Vdash^{\mc{M}}_\mrm{csv}\psi\ra \chi
     :\Leftrightarrow \forall \mc{M}', f [\mc{I}_{\mc{M}}\leqq_f \mc{I}_{\mc{M}'} \Ra
     f(w)\Vdash_{\mrm{G}, \mc{M}}\psi\ra\chi]$
     \item  $w\Vdash^{\mc{M}}_\mrm{csv}\forall v\psi :\Leftrightarrow  \forall n\in\omega (w\Vdash^{\mc{M}}_\mrm{csv} \psi(\bar n))$
     \item $w\Vdash^{\mc{M}}_\mrm{csv}\exists v\psi :\Leftrightarrow \forall w'\geq_\mc{M}w \, \exists n\in \omega (w'\Vdash^{\mc{M}}_\mrm{csv}\psi(\bar n))$
\end{enumerate}
    
\end{dfn}

\noindent $w\Vdash^{\mc{M}}_\mrm{csv}\vphi$  is to be read as $w$\textit{ csv-forces} $\vphi$ given $\mc{M}$. The semantic scheme given by the csv-forcing relation is labeled the CSV scheme. As before, the jump operator $\jcsv: \mc{P}(\mrm{SENT}_{\lt}) \longrightarrow \mc{P}(\mrm{SENT}_{\lt})$ is defined as:

\begin{center}
$\mc{J}_\mrm{csv}(X):=\{\# \vphi \sth \forall \mc{M}, \forall w\in W_\mc{M}(\mc{I}_{\mc{M}} (w)=X \Rightarrow w\Vdash^{\mc{M}}_\mrm{csv}\vphi)\}$  
\end{center}

The monotonicity of $\jcsv$ is nonetheless slightly harder to prove than it was for the previous jump operators. Before we get there, we prove a few properties of the CSV scheme. First, we can establish that csv-forcing implies global forcing:

\begin{prop}\label{intuitionistic soundness}
For any structure $\mc{M}$, world $w\in W_\mc{M}$, and formula $\vphi$ of $\lt$:  
\begin{center}
$w\Vdash^\mc{M}_\mrm{csv}\vphi \Ra w\Vdash_{\mrm{G}, \mc{M}}\vphi$.
\end{center}
\end{prop}
\begin{proof}
We prove the contrapositive by induction on the complexity of $\vphi$. The atomic case follows by definition. For the inductive step, we examine a couple of connectives: 
\begin{itemize}
    \item If $\vphi$ is of the form $\psi\ra\chi$: since $\mc{I}_\mc{M}\leqq_{\mrm{id}}\mc{I}_\mc{M}$, $w\nVdash_{\mrm{G}, \mc{M}}\psi\ra \chi$ implies $w\nVdash^\mc{M}_\mrm{csv}\psi\ra\chi$.
    \item If $\vphi$ is of the form $\psi\vee\chi$: then there is $w'\geq_\mc{M}w$ s.t. $w'\nVdash_{\mrm{G},\mc{M}}\psi$ and $w''\geq_\mc{M}w$ s.t. $w''\nVdash_{\mrm{G},\mc{M}}\chi$. By IH, $w'\nVdash_{\mrm{csv}}^\mc{M}\psi$ and $w''\nVdash_{\mrm{csv}}^\mc{M}\chi$, whence $w\nVdash_{\mrm{csv}}^\mc{M}\psi\vee\chi$ follows.  
    \item If $\vphi$ is of the form $\psi\wedge\chi$: then we can assume WLOG that $w\nVdash_{\mrm{G},\mc{M}}\psi$. By IH, $w\nVdash_{\mrm{csv}}^\mc{M}\psi$, whence $w\nVdash_{\mrm{csv}}^\mc{M}\psi\wedge \chi$ follows.
\end{itemize}
\end{proof}
Since the csv-forcing relation only considers worlds upwards, i.e., worlds accessible from the given world and not worlds that access the given world, the following holds:

\begin{prop}\label{truncated-model-csv}
Let $\mc{M}$ be a structure, and $w, w'\in W_\mc{M}$ be such that $w'\leq_\mc{M} w$. Let $\mc{M}_{w'}$ be the structure truncated at $w'$. Then, for $\vphi$ a formula of $\lt$:
\begin{equation*}
    w\Vdash^{\mc{M}}_\mrm{csv} \vphi \Leftrightarrow w\Vdash^{\mc{M}_{w'}}_\mrm{csv}.
\end{equation*}
\end{prop}

\noindent We can also show that csv-forcing is hereditary:

\begin{prop}
For any structure $\mc{M}$, world $w\in W_\mc{M}$, and formula $\vphi$ of $\lt$: 
\begin{equation*}
    w\Vdash_\mrm{csv}^\mc{M}\vphi \Ra \forall w'\geq_\mc{M}w (w'\Vdash^\mc{M}_\mrm{csv}\vphi)
\end{equation*}
\end{prop}
\begin{proof}
Once again, by induction on the complexity of $\vphi$. The persistency requirement guarantees the base case, as well as the case of $\forall$. The case of $\wedge$ follows by IH. The cases of $\vee$ and $\exists$ are straightforward from their respective definitions. 

Finally, for the case in which $\vphi$ is of the form $\psi\ra\chi$: let $w'\geq_\mc{M}w$ be such that $w'\nVdash_\mrm{csv}^\mc{M}\psi\ra\chi$. Then there are $\mc{M}'$ and $f$ with $\mc{I}_\mc{M}\leqq_f \mc{I}_\mc{M'}$ and $f(w')\nVdash_{\mrm{G},\mc{M}}\psi\ra\chi$. That is, there is $v\geq_\mc{M'}f(w')$ with $v\Vdash_{\mrm{G},\mc{M}}\psi$ and $v\nVdash_{\mrm{G},\mc{M}}\chi$. Since $f$ preserves the accessibility relation, $f(w)\leq_\mc{M'}v$, and so $f(w)\nVdash_{\mrm{G},\mc{M}}\psi\ra\chi$ follows, whence $w\nVdash_\mrm{csv}^\mc{M}\psi\ra\chi$.
\end{proof}

\noindent Now comes the crucial lemma for proving monotonicity: 

\begin{lemma}\label{monotonicity lemma}
Let $X\subseteq \mrm{SENT}_{\lt}$, let $\mc{M}$ be a structure, let $w$ be a world in $W_\mc{M}$ with $\mc{I}_\mc{M}(w)=X$, and let $\vphi$ be a formula of $\lt$ such that $w\nVdash^\mc{M}_\mrm{csv}\vphi$. Then, for all $X'\subseteq \mrm{SENT}_{\lt}$ with $X\subseteq X'$, there is a structure $\mc{M}'$ such that $w\in  W_\mc{M'}$ and $\mc{I}_\mc{M'}(w)=X'$ and $w\nVdash^\mc{M'}_\mrm{csv}\vphi$.
\end{lemma}

\begin{proof}
Once more, by induction on the complexity of $\vphi$. We examine the base case and a few relevant cases from the inductive step:
\begin{itemize}
    \item If $\vphi$ is atomic: consider the structure $\mc{M}'$ that results from taking the truncated structure $\mc{M}_w$ and letting $\mc{I}_{\mc{M}'}(w)=X'$. The structure remains a persistent, first-order intuitionistic Kripke $\omega$-structure, and obviously $w\nVdash^\mc{M'}_\mrm{csv}\vphi$.
    \item If $\vphi$ is $\psi\vee\chi$: $w\nVdash^\mc{M}_\mrm{csv}\psi\vee\chi$ implies the existence of $w'\leq_\mc{M}$ and $w''\leq_\mc{M}w$ s.t. $w'\nVdash^\mc{M}_\mrm{csv}\psi$ and $w''\nVdash^\mc{M}_\mrm{csv}\chi$. By IH, and since $X'\subseteq X\subseteq \mc{I}_{\mc{M}_w}(i)$ for $i$ any of $w', w''$, there are structures $\mc{M}_0$ and $\mc{M}_1$ with: (i) $w'\in W_{\mc{M}_0}$ and $w''\in W_{\mc{M}_1}$, (ii) $\mc{I}_{\mc{M}_0}(w')=X'$ and $\mc{I}_{\mc{M}_w}(w)=X'$, and (iii) $w'\nVdash^{\mc{M}_0}_\mrm{csv}\psi$ and $w''\nVdash^{\mc{M}_1}_\mrm{csv}\chi$. In order to obtain our structure $\mc{M}'$, take the truncated structures ${\mc{M}_0}_{w_0}$ and ${\mc{M}_1}_{w_1}$ and unite them with a root which we label $w$, letting $\mc{I}_{\mc{M}'}(w)=X'$. It is easy to see that the hereditary conditions are met. Finally, by Proposition \ref{truncated-model-csv}, 
    \begin{equation*}
        w'\nVdash^{\mc{M}_0}_\mrm{csv}\psi\Lra w'\nVdash^{\mc{M}'}_\mrm{csv}\psi
    \end{equation*}
    \begin{equation*}
        w''\nVdash^{\mc{M}_1}_\mrm{csv}\chi\Lra w''\nVdash^{\mc{M}'}_\mrm{csv}\chi
    \end{equation*}
    \noindent and so (iii) implies $w\nVdash^{\mc{M}'}_\mrm{csv}\psi\vee\chi$ by definition.
    \item If $\vphi$ is of the form $\psi\wedge\chi$: from $w\nVdash^\mc{M}_\mrm{csv}\psi\wedge\chi$ we can assume, WLOG, $w\nVdash^\mc{M}_\mrm{csv}\psi$. Then the claim follows easily by IH. 
    \item If $\vphi$ is of the form $\psi\ra\chi$: $w\nVdash^\mc{M}_\mrm{csv}\psi\ra\chi$ implies that there are $\mc{M}_0$ and $f$ s.t. $\mc{I}_\mc{M}\leqq_f\mc{I}_{\mc{M}_0}$ and $f(w)\nVdash_{\mrm{G}, \mc{M}_0}\psi\ra\chi$. Evidently, taking $\mc{M}_w$, and restricting $f$ to the worlds in it, $\mc{I}_\mc{M}\leqq_f\mc{I}_{\mc{M}_0}$, and $f(w)\nVdash_{\mrm{G}, \mc{M}_0}\psi\ra\chi$. We then let $\mc{M}'$ be like $\mc{M}_w$ but with $\mc{I}_\mc{M'}(w)=X'$, and we can conclude $w\nVdash^\mc{M'}_{\mrm{csv}}\psi\ra\chi$, since $\mc{I}_\mc{M'}\leqq_f\mc{I}_{\mc{M}_0}$.
\end{itemize}
\end{proof}

\begin{prop}
The following holds:
\begin{itemize}
    \item[i)] $\mc{J}_\mrm{csv}(\cdot)$ is monotonic. Therefore, there are fixed points of $\mc{J}_\mrm{csv}(\cdot)$.
    \item[ii)] There exist non-degenerate fixed-points of  $\mc{J}_\mrm{csv}(\cdot)$.  
\end{itemize}
\end{prop}

\begin{proof}
For i): take $X,Y$ such that $X\subseteq Y$, and assume $\#\vphi\notin \jcsv(Y)$. Then, there is some structure $\mc{M}$ and world $w\in W_\mc{M}$ such that $\mc{I}_\mc{M}(w)=Y$ and $w\nVdash^{\mc{M}}_\mrm{csv}\vphi$. By Lemma \ref{monotonicity lemma}, there is a structure $\mc{M}'$ such that $w\in W_{\mc{M}'}$, $\mc{I}_{\mc{M}'}(w)=X\subseteq Y$, and  $w\nVdash^{\mc{M'}}_\mrm{csv}\vphi$. Hence, $\#\vphi\notin\jcsv (X)$.

For ii): immediate, since any interpretation embeds into itself.
\end{proof}

\noindent Having proved monotonicity, we now want to establish the compositional properties of fixed-points of $\mc{J}_\mrm{csv}$. For that purpose, we first show that a world csv-forces a formula exactly when that formula is csv-forced, by the corresponding world, at all embeddable structures. 

\begin{lemma}\label{lemma1}
For all structures $\mc{M}$, worlds $w$ and formulae $\vphi$:

\begin{center}
    $w\Vdash^{\mc{M}}_\mrm{csv} \vphi \Leftrightarrow \forall \mc{M}', f(\mc{I}_\mc{M}\leqq_f \mc{I}_{\mc{M}'} \Ra f(w)\Vdash^{\mc{M}'}_\mrm{csv} \vphi)$
\end{center}
\end{lemma}

\begin{proof}
\textbf{Left-to-right}: by induction on the complexity of $\vphi$. 

\noindent If $\vphi$ is atomic, i.e., of the form $R(t_1,...,t_n)$. The claim follows immediately by the transitivity property of the interpretation-embedding relation.\\
\noindent If $\vphi$ is of the form $\psi\vee\chi$, we can reason as follows:
\begin{align*}
    w\Vdash^\mc{M}_\mrm{csv}\psi\vee\chi & \Ra \forall w'\geq_\mc{M} w(w'\Vdash^\mc{M}_\mrm{csv}\psi) \text{ or } \forall w'\geq_\mc{M} w(w'\Vdash^\mc{M}_\mrm{csv}\chi) & \\
    & \Ra \forall w'\geq_\mc{M} w(\forall \mc{M'}, f (\mc{I}_\mc{M}\leqq_f\mc{I}_\mc{M'} \Ra f(w')\Vdash^\mc{M'}_\mrm{csv}\psi)) \text{ or }& \\ 
    &\forall w'\geq_\mc{M} w(\forall \mc{M'}, f (\mc{I}_\mc{M}\leqq_f\mc{I}_\mc{M'} \Ra f(w')\Vdash^\mc{M'}_\mrm{csv}\chi)) & \text{ by IH} \\
\end{align*}
\noindent Assume WLOG that the first disjunct holds. By hereditariness of the csv-forcing, for any $\mc{M}'$ and $f$ s.t. $\mc{I}_\mc{M}\leqq_f\mc{I}_\mc{M'}$ and any $w'\geq_\mc{M'}f(w)$, $w'\Vdash^\mc{M'}_\mrm{csv}\chi$. Therefore, 
\begin{equation*}
    \forall \mc{M}', f(\mc{I}_\mc{M}\leqq_f \mc{I}_{\mc{M}'} \Ra f(w)\Vdash^{\mc{M}'}_\mrm{csv} \psi\vee\chi).
\end{equation*}
The case of $\vphi$ being of the form $\exists v \psi$ follows similarly. When $\vphi$ is of the form $\psi\wedge\chi$ or $\forall v \psi$, the claim obtains easily by IH.\\
\noindent Finally, we consider the case in which $\vphi$ is of the form $\psi\ra\chi$. Then our assumption yields:
\begin{equation*}
    \forall \mc{M}', f(\mc{I}_\mc{M}\leqq_f\mc{I}_\mc{M'}\Ra f(w)\Vdash_\mc{M'}\psi\ra\chi)
\end{equation*} 
Taking any $\mc{M}'$, $f$ such that $\mc{I}_\mc{M}\leqq_f\mc{I}_\mc{M'}$ we aim to show $f(w)\Vdash^\mc{M'}_\mrm{csv}\psi\ra\chi$. So we take a structure $\mc{N}$ and a function $g$ s.t. $\mc{I}_\mc{M'}\leqq_g\mc{I}_\mc{N}$; we want to establish that $g(f(w))\Vdash_{\mrm{G}, \mc{M}}\psi\ra\chi$. By transitivity of the embedding function, $\mc{I}_\mc{M}\leqq_{g\circ f}\mc{I}_\mc{N}$; so, from our assumption, it follows that $g(f(w))\Vdash_{\mrm{G}, \mc{M}}\psi\ra\chi$.

\textbf{Right-to-left}: Assume $\forall \mc{M}', \forall f(\mc{I}_\mc{M}\leqq_f \mc{I}_{\mc{M}'} \Ra f(w)\Vdash^{\mc{M}'}_\mrm{csv} \vphi)$. Then, since $\mc{I}_\mc{M}\leqq \mc{I}_{\mc{M}}$ trivially via the identity embedding, $w\Vdash^{\mc{M}}_\mrm{csv} \vphi $. 
\end{proof}
\noindent This result suffices to prove the desired compositional properties of our theory. In the case of disjunctions and existentially quantified sentences, they key consists in `zooming-in' to the one-world model, which embeds into any other relevant structure.

\begin{lemma}\label{lemma disjunction}
Let $X=\mc{J}_\mrm{csv}(X)$. Then: $\#\vphi\vee\psi\in X \Ra \#\vphi \in X\vee \#\psi\in X$.
\end{lemma}
\begin{proof}
If $\#\vphi\vee\psi\in X$, then by the fixed-point property of $X$, 

\begin{center}
    $\forall \mc{M}, w (\mc{I}_\mc{M}(w)=X \Ra w\VdashM \vphi\vee\psi)$
\end{center}

Let $\mc{M}$ be the one-world model given by $W_\mc{M}=\{w\}$ and $\mc{I}_\mc{M}=\{\langle w,x\rangle \sth x\in X\}$. Then $w\VdashM \vphi\vee\psi$, so $w\VdashM\vphi$ or $w\VdashM\psi$. Assume WLOG that $w\VdashM\vphi$. We will show $\#\vphi\in X$, for which we require $\forall \mc{M}, w (\mc{I}_\mc{M}(w)=X \Ra w\VdashM \vphi)$. So take $\mc{M}', w'$ s.t. $\mc{I}_{\mc{M}'}(w')=X$, and we want to show  $w'\Vdash^{\mc{M}'}_\mrm{csv} \vphi$. By Lemma \ref{lemma1}, it's enough to prove 

\begin{equation}\label{eqn2}
    \forall \mc{N}, f' (\mc{I}_{\mc{M}'}\leqq_{f'} \mc{I}_\mc{N}\Ra f'(w') \Vdash^{\mc{N}}_\mrm{csv} \vphi)
\end{equation}
Hence, we take an arbitrary $\mc{N}_0$, $f_0$ such that $\mc{I}_{\mc{M}'}\leqq_{f_0} \mc{I}_{\mc{N}_0}$.

First, we note that $\mc{M}$ embeds into $\mc{M}'$ (since it embeds into any structure with a world whose truth-interpretation is greater or equal than $X$); and, by the transitivity of the embedding, there is an embedding function $f_1$ s.t. $\mc{I}_\mc{M}\leqq_{f_1} \mc{I}_{\mc{N}_0}$ and $f_0(w')=f_1(w)$. Since we assumed $w\VdashM\vphi$, Lemma \ref{lemma1} yields $f_0(w')=f_1(w) \Vdash^{\mc{N}}_\mrm{csv} \vphi$. This completes the proof of (\ref{eqn2}), and so $\#\vphi\in X$.
\end{proof}

\noindent For the next lemma, the reasoning is exactly like in the previous lemma. The key is to note that the one-world structure embeds into any other structure, and that the embedding relation is transitive. It's also essential to note that the domain of all worlds across all structures is always $\omega$, so that the witness of an existential formula in the one-world model can also be a witness in any other world of a structure that the one-world structure embeds into. 

\begin{lemma}\label{lemma existential}
Let $X=\jcsv(X)$. Then: $\#\exists v\vphi\in X \Ra \exists n\in \omega \#\vphi(\bar n) \in X$.
\end{lemma}

\noindent Although not strictly compositional, the following property is often seen as an extension of compositionality for the truth predicate:

\begin{lemma}\label{T-rep and T-Del}
Let $X=\jcsv(X)$. Then: $\#\T\ulcorner\vphi\urcorner \in X \Leftrightarrow \#\vphi\in X$.
\end{lemma}

\begin{proof}
In both directions, the claim follows from the fixed-point property and the clause for atomic sentences of the relation $\VdashM$.
\end{proof}

\noindent Summing up:

\begin{thm}\label{compositional clauses}
Let $X=\jcsv(X)$. Let $\mc{M}$ be a structure and $w$ a world in $W_\mc{M}$ such that $\mc{I}_\mc{M}(w)=X$. Then:

\begin{enumerate}
    \item $w\Vdash_{\mrm{G},\mc{M}} \forall x,y (\mrm{Sent}_{\lt}(x\subdot \vee y) \ra (\T (x\subdot \vee y)\leftrightarrow \T x\vee \T y))$
    \item $w\Vdash_{\mrm{G},\mc{M}} \forall x,y (\mrm{Sent}_{\lt}(x\subdot \wedge y)\ra (\T (x\subdot \wedge y)\leftrightarrow \T x\wedge \T y))$
     \item $w\Vdash_{\mrm{G},\mc{M}} \forall x,v (\mrm{Sent}_{\lt}(\exists v x)\ra (\T (\subdot \exists v x)\leftrightarrow \exists t \T x(t/v)))$
     \item $w\Vdash_{\mrm{G},\mc{M}} \forall x,v (\mrm{Sent}_{\lt}(\forall v x)\ra (\T (\subdot \forall v x)\leftrightarrow \forall t \T x(t/v)))$
     \item $w\Vdash_{\mrm{G},\mc{M}} \forall t(\T \subdot \T t\leftrightarrow \T t^\circ )$
    
\end{enumerate}
\end{thm}

\begin{proof}

For 1: one needs to show $\#\vphi\vee\psi\in X \Leftrightarrow (\#\vphi\in X \vee \#\psi\in X)$. The left-to-right direction is Lemma \ref{lemma disjunction}; the right-to-left is immediate by the definition of $\VdashM$.

\noindent For 2:  one needs to show $\#\vphi\wedge\psi\in X \Leftrightarrow (\#\vphi\in X \wedge \#\psi\in X)$, but both directions are immediate by the definition of $\VdashM$.

\noindent For 3: one needs to show $\#\exists v\vphi\in X \Leftrightarrow \exists n\in \omega \#\vphi(\bar n) \in X$. The left-to-right direction is provided by Lemma \ref{lemma existential}; the right-to-left is immediate by the definition of $\VdashM$.

\noindent For 4:  one needs to show $\#\forall v\vphi\in X \Leftrightarrow \forall n\in \omega \#\vphi(\bar n) \in X$, but both directions are immediate by the definition of $\VdashM$.

\noindent For 5: one needs to show $\#\T\ulcorner\vphi\urcorner \in X \Leftrightarrow \#\vphi\in X$, which is just Lemma \ref{T-rep and T-Del}.

\end{proof}

\noindent One connective is not considered above: the conditional. Although full compositionality for it cannot be expected, one direction thereof (namely, the possibility of distributing truth over the conditional) does obtain. But before proving it, we shall establish the supervaluational character of our theory. This will allow us to show, among others, that HA `lives' inside our supervaluational theory. 

\begin{lemma}\label{csv is supervaluational}
Let $X\subseteq \mrm{SENT}_{\lt}$. Then $\forall \mc{M},w(\mc{I}_\mc{M}(w)=X\Ra w\Vdash_{\mrm{G},\mc{M}}\vphi)$ implies $\#\vphi \in \mc{J}_\mrm{csv}(X)$.
\end{lemma}

\begin{proof}
We establish the contrapositive. Assume $\#\vphi\notin \mc{J}_\mrm{csv}(X)$.  We proceed by induction on the complexity of $\vphi$. The base case is immediate. 
\noindent For the inductive step, consider first the case in which $\vphi$ is of the form $\psi\vee\chi$. If $\#\psi\vee\chi\notin \mc{J}_\mrm{csv}(X)$, there is a structure $\mc{M}$ and a world $w$ with $\mc{I}_\mc{M}(w)=X$ and $w\nVdash^\mc{M}_\mrm{csv}\psi\vee\chi$. Hence, there are $w'\geq_\mc{M} w$ and $w''\geq_\mc{M} w$ s.t. $w'\nVdash^\mc{M}_\mrm{csv}\psi$ and $w''\nVdash^\mc{M}_\mrm{csv}\chi$. Suppose $\mc{I}_\mc{M}(w')=X'$ and $\mc{I}_\mc{M}(w'')=X''$. By IH, there are $\mc{M'}$ and $\mc{M}''$ with $w'\in W_\mc{M'}$ and $w''\in W_\mc{M''}$, and $\mc{I}_\mc{M'}(w')=X'$ and $\mc{I}_\mc{M''}(w'')=X''$, and $w'\nVdash_{\mrm{G}, \mc{M'}}\psi$ and $w''\nVdash_{\mrm{G}, \mc{M''}}\chi$. Build a new structure $\mc{N}$ by taking $\mc{M}'_{w'}$ and $\mc{M}'_{w''}$ and uniting them with a root $w$ s.t. $\mc{I}_\mc{N}(w)=X$ (which we can do because $X', X''\supseteq X$). Then, by the clauses for global forcing, $w\nVdash_{\mrm{G}, \mc{N}}\psi\vee\chi$.

Take now the case in which $\vphi$ is of the form $\psi\ra\chi$. $\#\psi\ra\chi\notin \mc{J}_\mrm{csv}(X)$ implies that there is a structure $\mc{M}$ and a world $w$ with $\mc{I}_\mc{M}(w)=X$ and $w\nVdash^\mc{M}_\mrm{csv}\psi\ra\chi$. Hence, there is a structure $\mc{M}'$ and a function $f$ s.t. $\mc{I}_\mc{M}\leqq_f\mc{I}_\mc{M'}$ and $f(w)\nVdash_{\mrm{G},\mc{M'}}\psi\ra\chi$. We now build a new structure, $\mc{N}$, which takes $\mc{M}_{f(w)}$ and adds a root, $w$, with $\mc{I}_\mc{N}(w)=X$. This structure meets the hereditary conditions and, clearly, $w\nVdash_{\mrm{G},\mc{N}}\psi\ra\chi$. 

The case of the existential quantifier is similar to that of disjunction, and the cases for conjunction and the universal quantifier follow easily by IH. 
\end{proof}

\noindent Since all worlds of all structures force the axioms of HAT, we obtain a first, important result that hints at the real supervaluational character of this theory:

\begin{prop}\label{fp models are supervaluational}
Let $X\subseteq \mrm{SENT}_{\lt}$. Let $\mc{M}$ be a structure and $w$ a world in $W_\mc{M}$ such that $\mc{I}_\mc{M}(w)=\jcsv(X)$. Then:
\begin{center}
    $w\Vdash_{\mrm{G},\mc{M}} \forall x(\mrm{Ax}_{\mrm{HAT}}(x)\rightarrow \T (x))$
\end{center}

\end{prop}

\noindent This being said, there are further important consequences of Lemma \ref{csv is supervaluational}. First: 

\begin{corollary}\label{csv iff classical}
For any $X\subseteq \mrm{SENT}_{\lt}$ and any formula $\vphi$:
\begin{center}
$\forall \mc{M},w(\mc{I}_\mc{M}(w)=X\Ra w\Vdash_{\mrm{G},\mc{M}}\vphi)$ iff $\#\vphi \in \mc{J}_\mrm{csv}(X)$.
\end{center}
\end{corollary}
\begin{proof}
The right-to-left direction follows simply from the fact that $w\VdashM\vphi$ implies $w\Vdash_{\mrm{G},\mc{M}}\vphi$ (Proposition \ref{intuitionistic soundness}).
\end{proof}

\noindent Corollary \ref{csv iff classical} \textit{really is} the core of a supervaluational semantic relation, so we can already claim that CSV is such a relation. But we can convince ourselves further that the theory of truth we are constructing is supervaluational: we can show that the standard supervaluational jump, as defined in section \ref{semantics}, replicated over G-forcing, coincides with the CSV scheme. To that effect, we define:

\begin{dfn}
Let $\mc{M}$ be a structure, and let $w\in W_\mc{M}$. For a formula $\vphi$ of $\lt$, we define the relation $w\Vdash^{\mc{M}}_{\mrm{svi}_2}\vphi$ as follows:

\begin{center}
    $w\Vdash^{\mc{M}}_{\mrm{svi}_2}\vphi: \Leftrightarrow \forall \mc{M}'\forall f(\mc{I}_\mc{M}\leqq_f \mc{I}_{\mc{M}'} \Rightarrow f(w)\Vdash_{\mrm{G}, \mc{M}'} \vphi)$.
\end{center}
\end{dfn}

\noindent We now write: 
\begin{equation*}
\mc{J}_{\mrm{svi}_2}(X):=\{\# \vphi \sth \forall \mc{M}, \forall w\in W_\mc{M}(\mc{I}_{\mc{M}} (w)=X \Rightarrow w\Vdash^{\mc{M}}_{\mrm{svi}_2}\vphi)\}
\end{equation*}

\noindent Once again, this is a strict supervaluational theory of truth in the style of Kripke, which operates just like SVI, with the exception that the underlying forcing relation is not the usual one but the global one. Then:

\begin{prop}\label{svi and csv}
$\#\vphi\in \mc{J}_{\mrm{svi}_2}(X) \Leftrightarrow  \#\vphi\in \mc{J}_\mrm{csv}(X)$ for any set of codes of sentences $X$. 
\end{prop}

\begin{proof}

It's relatively straightforward to show that $\forall \mc{M},w(\mc{I}_\mc{M}(w)=X\Ra w\Vdash_{\mrm{G},\mc{M}}\vphi)$ iff $\#\vphi\in \mc{J}_{\mrm{svi}_2}(X)$. Therefore, the claim follows from Corollary \ref{csv iff classical}.
\end{proof}

\noindent The outtake from this is then the following: the theory we have constructed can be given using the compositional supervaluations, or using a more traditional, clearly supervaluational jump. Thus, we can assert that we have a supervaluational theory that respects compositionality for disjunction, conjunction, and the existential and the universal quantifiers. Moreover, it is now straightfoward to show that the theory also distributes over the conditional:

\begin{prop}\label{distribution conditional}
Let $X\subseteq \mrm{SENT}_{\lt}$ be such that $X=\mc{J}_\mrm{csv}(X)$. Then, for any formulae $\psi,\chi$, if $\#\psi\ra\chi\in X$ and $\#\psi\in X$, then $\#\chi \in X$. 
\end{prop}
\begin{proof}
We make use of the equivalence in Corollary \ref{csv iff classical}. From this corollary and our assumptions, we get $\forall \mc{M}, w\in W_\mc{M}(\mc{I}_\mc{M}(w)=X\Ra w\Vdash_{\mrm{G}, \mc{M}}\psi\wedge\psi\ra\chi)$. Taking any such $\mc{M}'$ and $w'\in W_\mc{M'}$, $w'\Vdash_{\mrm{G}, \mc{M'}}\psi$ follows immediately. Since these were an arbitrary structure and world, $\#\chi \in \mc{J}_{\mrm{csv}}(X)$.
\end{proof}
\begin{corollary}\label{consistency of fps}
For any formula $\vphi$ of $\lt$ and for any set of codes of sentences $X$, it is not the case that $\#\vphi\in \jcsv(X)$ and $\#\neg\vphi\in \jcsv(X)$.
\end{corollary}
\begin{proof}
By Corollary \ref{csv iff classical}.
\end{proof}

\noindent Finally, we can present the fixed-point models of the theory CSV, and the axiomatic theory that they sanction. 
\begin{dfn}
A \emph{fixed-point model} of $\mrm{CSV}$ is a persistent first-order intuitionistic Kripke $\omega$-structure $\mc{M}:=\langle W_\mc{M}, \leq_\mc{M}, \mc{D}_\mc{M}, \mc{I}_\mc{M} \rangle$ such that, for all $X\subseteq \mrm{SENT}_{\lt}$ and all worlds $w\in W_\mc{M}$, if $\mc{I}_\mc{M}(w)=X$, then $X=\jcsv(X)$.
\end{dfn}

\begin{thm}\label{axioms of CSV}
Let $\mc{M}$ be a fixed-point model of $\mrm{CSV}$. Then $\mc{M}$ satisfies all of the following sentences, which we call the \emph{axioms of} $\mrm{CSV}$:

\begin{enumerate}
    \item $\forall s,t((\T(s\subdot{=}t) \leftrightarrow s^{\circ} = t^{\circ})\wedge (\T(s\subdot{\neq}t) \leftrightarrow s^{\circ} \neq t^{\circ}))$
    \item $\forall x,y (\mrm{Sent}_{\lt}(x\subdot \vee y) \ra (\T (x\subdot \vee y)\leftrightarrow \T x\vee \T y))$
    \item $\forall x,y (\mrm{Sent}_{\lt}(x\subdot \wedge y)\ra (\T (x\subdot \wedge y)\leftrightarrow \T x\wedge \T y))$
     \item $\forall x,v (\mrm{Sent}_{\lt}(\exists v x)\ra (\T (\subdot \exists v x)\leftrightarrow \exists t \T x(t/v)))$
     \item $\forall x,v (\mrm{Sent}_{\lt}(\forall v x)\ra (\T (\subdot \forall v x)\leftrightarrow \forall t \T x(t/v)))$
     \item $\forall x,y (\mrm{Sent}_{\lt}(x\subdot\ra y)\ra(\T ( x\subdot \ra y)\rightarrow (\T x\ra \T y )))$
      \item $\forall t(\T t^\circ \leftrightarrow \T\subdot\T t)$
     \item $\forall x(\mrm{Ax}_{\mrm{HAT}}(x)\rightarrow \T x)$ 
     \item $\forall x(\T x\wedge \T \subdot \neg x \rightarrow \bot)$ 
     \item$\forall x(\T\ulcorner \T(\dot{x}) \to \mathrm{Sent}(\dot{x})\urcorner)$
     \item $\T\ulcorner \vphi\urcorner \rightarrow \vphi$ for any formula $\vphi$ of $\lt$
\end{enumerate}
\end{thm}
\begin{proof}
2-5 and 7 follow from Theorem \ref{compositional clauses}; 6 follows Proposition \ref{distribution conditional}. 8 follows from Proposition \ref{fp models are supervaluational}. 1 is an immediate consequence of Corollary \ref{csv iff classical}. 11 follows by the soundness of the csv-forcing relation with respect to global forcing, i.e., Proposition \ref{intuitionistic soundness}.

\noindent For 9: by Corollary \ref{consistency of fps}, all fixed-points are consistent. Therefore, taking any $w\in W_\mc{M}$, $w'\nVdash_{\mrm{G},\mc{M}} \T x \wedge \T \subdot \neg x$ for all $x$ and all $w'\geq w$. Hence,  $w\Vdash_{\mrm{G},\mc{M}} \T x \wedge \T \subdot \neg x \ra \bot$. 

\noindent For 10: it follows from the fact that the jump $\jcsv$ is only defined on sets of codes of sentences.

\end{proof}

\noindent Therefore, we have proved the consistency of a supervaluational theory that is fully compositional for all connectives except for the conditional. Finally, it must be noted that the proof-theoretic strength of $\mrm{CSV}$ was basically established by Leigh: his theory $\mrm{H}^i$ in \cite{leigh_2013} is essentially the same as $\mrm{CSV}$.

\begin{prop}[Leigh]
Let $\mrm{CSV}$ be the theory extending $\mrm{HAT}$ with the axioms of CSV as defined in Theorem \ref{axioms of CSV}. Then, $\mrm{CSV}$ is consistent. Moreover, $\vert \mrm{CSV}\equiv \vert \mrm{ID}^i_1\vert$, i.e., $\mrm{CSV}$ proves the same  $\lnat$-theorems as the intuitionistic theory of one inductive definition. 
\end{prop}
\begin{proof}
The point is just that of Proposition \ref{proof theory ISV}: Leigh shows that his theory $\mrm{H}^i$ can: 
\begin{itemize}
    \item derive every $\lnat$-theorem of $\mrm{ID}^i_1$.
    \item be interpreted in the intuitionistic version of Kripke-Platek with Infinity, $\mrm{KP}^i_\omega$, which is proof-theoretically equivalent to $\mrm{ID}^i_1$.
\end{itemize}

\noindent Thus, the upper bound for $\mrm{CSV}$ is given by ii), and an inspection of the proof of i) shows that all the resources to carry out the proof are already present in $\mrm{CSV}$.  
\end{proof}

\section{Conclusion}

In this paper, we have introduced supervaluationist theories of truth over intuitionistic logic. To that effect, we have shown how to do supervaluations for intuitionistic logic and, in particular, how to do supervaluations over Kripke structures. In addition, we have paired the semantic theory thus constructed with an axiomatic theory that mimicked Cantini's celebrated theory VF. As we mentioned in the introduction, a project like ours yields a potential formal theory of truth for the constructive mathematician, one who does not accept classical logic but still thinks that all intuitionistic validities are true. We have also compared the particular method we employ, based on interpretation embeddings and ranging over possibly greater structures, with the one that has often been used in the literature to produce theories of truth and or necessity with possible-world semantics.

Furthermore, we aimed to fix a common problem of supervaluational truth theories over classical logic: the failure of compositionality for disjunctions and existentially quantified sentences. We addressed this issue by presenting a semantic scheme that was compositionally delivered, but which, has we eventually showed, corresponds to the basic format of supervaluations over a different forcing relation. The outcome is a theory of truth that not only delivers transparent fixed-points, but that is at the same time intuitionistically supervaluational (thus recovering all validities of intuitionistic logic) and enjoys a greater degree of compositionality than the supervaluational theories over classical logic.

\footnotesize \subsection*{\small Acknowledgements} We would like to thank Jonathan Beere, Beatrice Buonaguidi, Leon Horsten, Mateus \L e\l yk and Karl Georg Niebergall, as well as two anonymous referees, for helpful comments on different versions of this paper. This work was made possible by an LAHP (London Arts and Humanities Partnership) studentship, as well as by PLEXUS (Grant Agreement no 101086295), a Marie Sklodowska-Curie action funded by the EU under the Horizon Europe Research and Innovation Programme.

\normalsize

\bibliography{bibliography.bib}
\bibliographystyle{apalike}

\begin{appendices}
\renewcommand{\thesection}{\arabic{section}}
\section{Proof of Theorem \ref{completeness of G}}\label{appendix one}

In order to state the exact theorem that we shall prove, we provide a couple of definitions. Most of them are simple generalizations of the ones we already presented.

\begin{dfn}[$\mc{L}_\mrm{IPC}$]
The language $\mc{L}_\mrm{IPC}$ is the first-order language with logical constants $\wedge, \vee, \ra, \forall, \exists, \bot$, countably many variables from a set $V=\{v_1,...,v_n,...\}$, countably many predicates from a set $P=\{P_1,...,P_n,...\}$, countably many constants from a set $C=\{c_1,...,c_n,...\}$, and countably many functions from a set $F=\{f_1,...,f_n,...\}$. 
\end{dfn}

\noindent In an abuse of notation, we write $\mc{L}\subseteq \mc{L}_\mrm{IPC}$ to denote that $\mc{L}$ is a fragment of $\mc{L}_\mrm{IPC}$ which contains all the logical connectives of $\mc{L}_\mrm{IPC}$. In particular, $\lt\subseteq \mc{L}_\mrm{IPC}$. For a given set $X$, we also denote with $\langle X\rangle^{<\omega}$ the set of finite tuples of $X$. 

\begin{dfn}
A persistent first-order Kripke model $\mc{M}$ for a first-order language $\mc{L}\subseteq \mc{L}_\mrm{IPC}$ is a quadruple $\langle W_\mc{M}, \leq_\mc{M}, \mc{D}_\mc{M}, \mc{I}_\mc{M}\rangle$ such that

\begin{itemize}
    \item $W_\mc{M}$ is a set of worlds.
    \item $\leq_\mc{M}$ is an accessibility relation on $W_\mc{M}$ which is reflexive and transitive.
    \item $\mc{D}_{\mc{M}}$ is a function assigning elements of $W_\mc{M}$ to non-empty sets such that $w\leq_\mc{M}w'$ implies $\mc{D}_{\mc{M}}(w)\subseteq \mc{D}_{\mc{M}}(w')$ for all $w, w'\in W_\mc{M}$. The set $\bigcup \{\mc{D}_\mc{M}(w) \sth w\in W_\mc{M}\}$ is the \emph{domain} of $\mc{M}$, denoted simply $\bigcup \mc{D}_\mc{M}$.
    \item $\mc{I}_\mc{M}$ is a triple $\langle \mc{M}_P, \mc{M}_C, \mc{M}_F \rangle$ such that:
    \begin{itemize}
        \item $\mc{M}_P$ is a subset of $ P\times W_\mc{M} \times \langle \bigcup \mc{D}_\mc{M}\rangle^{<\omega}$, with the requirements that (i) $\langle P_i, w, \langle t_1,...,t_n\rangle\rangle \in \mc{M}_P$ iff $P_i\in \mc{L}$ and $t_1,...,t_n\in \mc{D}_\mc{M}(w)$ and $P_i$ is an $n$-ary predicate; and (ii) $\langle P_i, w, \langle t_1,...,t_n\rangle\rangle \in \mc{M}_P$ implies $\langle P_i, w', \langle t_1,...,t_n\rangle\rangle \in \mc{M}_P$ for all $P_i$, $t_1,...,t_n$, and $w,w'\in W_\mc{M}$ such that $w\leq_\mc{M}w'$.
        \item $\mc{M}_C$ is a function assigning at most one element of $\mc{D}_\mc{M}(w)$ to the pair $\langle w, c_i\rangle$, for $w\in W_\mc{M}$ and $c_i\in\mc{L}$. We also require that $\mc{M}_C(w, c_i)=t_i$ implies $\mc{M}_C(w', c_i)=t_i$ for all $w,w'\in W_\mc{M}$ such that $w\leq_\mc{M}w'$.
        \item $\mc{M}_F$ is a function assigning a function $f'_i: \mc{D}_\mc{M}(w)^n \longrightarrow \mc{D}_\mc{M}(w)$ to the pair $\langle w, f_i\rangle$, for $f_i\in F$ and $w\in W_\mc{M}$, provided $f_i$ is an $n$-ary function. We further require that $\mc{M}_C(w, f_i)(t_1,...,t_n)=t_m$ implies $\mc{M}_C(w', c_i)(t_1,...,t_n)=t_m$ for all $w,w'\in W_\mc{M}$ such that $w\leq_\mc{M}w'$.
    \end{itemize}
\end{itemize}
\end{dfn}

\noindent We now define the value of a closed term:

\begin{dfn}\label{value of a term}
    Given a language $\mc{L}\subseteq \mc{L}_\mrm{IPC}$, if $t$ is a closed term of $\mc{L}$, and $\mc{M}$ is a persistent first-order Kripke model for $\mc{L}$, $w\in W_\mc{M}$, the \emph{value of} $t$, $t^\mc{M}$, is defined as follows:
    \begin{enumerate}
        \item $t^\mc{M} = \mc{M}_C(w,c_i)$ if $t$ is a constant $c_i$ of $\mc{L}$.
        \item $t^\mc{M}= \mc{M}_F(w, f)(t_1^\mc{M}, ..., t_n^\mc{M})$ if $t$ is of the form $f(t_1,...,t_n)$, where $t_1,...,t_n$ are terms and $f$ is an $n$-ary function symbol of $\mc{L}$. 
    \end{enumerate}
\end{dfn}

\noindent This suffices for understanding the clauses of global forcing in Definition \ref{global forcing}. On the basis of this, we want to prove:

\begin{thmmm}
Let $A$ be a formula of $\mc{L}_\mrm{IPC}$. Then:

\begin{equation*}
    \vdash_\mrm{IPC} A \Leftrightarrow \vDash_\mrm{G} A
\end{equation*}

\noindent where $\vDash_\mrm{G} A$ indicates that $A$ is $\mrm{G}$-forced at every world of every persistent first-order Kripke structure.

\end{thmmm}

\noindent The theorem, whose proof builds on \cite{dosen_1991} and \cite{fitting_1969}, will require a detour into the quantified modal logic S4.\footnote{Although quantified S4 is often denoted QS4 to distinguish it from the propositional S4, we will simply call it S4 here.} We first define a translation relating the language of modal logic and that of intuitionistic logic: 

\begin{dfn}[$\mc{L}_\Box$]
The language $\mc{L}_\Box$ extends the language $\mc{L}_\mrm{IPC}$ with the unary operator $\Box$. Thus, every formula of $\mc{L}_\mrm{IPC}$ is a formula of $\mc{L}_\Box$; and, if $\vphi$ is a formula of $\mc{L}_\Box$, then $\Box\vphi$ is a formula of $\mc{L}_\Box$.
\end{dfn}

\begin{dfn}[Translation $g$]
We define a translation $g:\mc{L}_\mrm{IPC}\longrightarrow \mc{L}_\Box$ as follows: 

\begin{align*}
    g(P(t_1,...,t_n)) & = \Box P(t_1,...,t_n), \, \text{for } P(t_1,...,t_n) \, \text{an atomic formula} \\
    g(\vphi \wedge\psi) & = g(\vphi)\wedge g(\psi)\\
    g(\vphi \vee\psi) & = g(\vphi)\vee  g(\psi)\\
    g(\vphi \ra\psi) & = \Box (g(\vphi)\ra g(\psi))\\
    g(\forall v\vphi) & = \Box \forall v g(\vphi(v))\\
    g(\exists v\vphi) & = \exists v g(\vphi(v))\\
    g(\bot) &= \bot
\end{align*}
\end{dfn} 

\noindent The translation $g$ is well-known, and features, e.g., in \cite[230]{schwichtenberg_troelstra_1996}.

Now, for the derivation relations $\vdash_\mrm{S4}$ and $\vdash_\mrm{IPC}$, we can work over natural deduction systems such as the ones defined in e.g. \cite{flagg_friedman_1986}. We also assume a standard presentation of quantified S4 models via Kripke semantics, with $w\Vdash_\mrm{S4}\vphi$ being the S4-forcing relation for $w$ a world in an S4 model and $\vphi$ a formula of $\lbox$, and $\vDash_\mrm{S4}$ being the satisfaction relation w.r.t. S4 models. As is often done, we take these models to be S4 models (i.e. Kripke models with a reflexive and transitive accessibility relation) with expanding domains, i.e., if $w, w'$ are worlds in a model $\mc{M}$ and $w\leq w'$, then $\mc{D}(w)\subseteq \mc{D}(w')$. It is well-known that quantified S4 is strongly complete with respect to the class of such Kripke models---see e.g. \cite{hughes_cresswell_1996}.

We first note that the G-forcing relation is hereditary, that is:
\begin{lemma}\label{hereditariness g}
Let $\mc{L}\subseteq \mc{L}_\mrm{IPC}$ and let $\mc{M}$ be a persistent first-order Kripke model for $\mc{L}$, with $w\in W_\mc{M}$. Then, if $w'\geq_\mc{M} w$,
\begin{equation*}
    w\Vdash_\mrm{G} \vphi \Ra w'\Vdash_\mrm{G} \vphi
\end{equation*}
\end{lemma}
\begin{proof}
We induct on the complexity of $\vphi$. The only distinct cases with respect to the usual hereditary result for intuitionistic logic are those of $\vee$ and $\exists$. We prove the former. Assume $w'\nVdash_\mrm{G} \psi\vee \chi$. Then, there is some $w''\geq_\mc{M}w'$ such that $w''\nVdash_\mrm{G} \psi$ and some $w'''\geq_\mc{M}w'$ such that $w'''\nVdash_\mrm{G} \chi$. Since $\geq_\mrm{G}$ is transitive, $w'''\geq_\mc{M}w$ and $w''\geq_\mc{M}w$, so $w\nVdash_\mrm{G} \psi\vee \chi$.
\end{proof}

\begin{lemma}\label{atomic sentences}
Let $\mc{M}=\langle W_\mc{M}, \leq_\mc{M}, \mc{D}_\mc{M}, \mc{I}_\mc{M}\rangle$ be a persistent Kripke first-order structure for the language $\lipc$, with $\mc{I}_\mc{M}$ being the triple $\langle \mc{M}_P, \mc{M}_C, \mc{M}_F\rangle$. Now, consider an $\lbox$ model $\mc{N}=\langle W_\mc{N}, \leq_\mc{N}, \mc{D}_\mc{N}, \mc{I}_\mc{N}\rangle$ given by $W_\mc{N}=W_\mc{M}, \leq_\mc{N}=\leq_\mc{M},  \mc{D}_\mc{N}=\mc{D}_\mc{M}$, and $\mc{I}_\mc{N}$ being such that $\mc{N}_C=\mc{M}_C$ and $\mc{M}_F=\mc{M}_F$ and, for all atomic $P(t_1,...,t_n)$:

    \begin{equation*}
    w\Vdash_{\mrm{G},\mc{M}}P(t_1,...,t_n) \Leftrightarrow w\Vdash_{\mrm{S4},\mc{N}}g(P(t_1,...,t_n))
    \end{equation*}

\noindent Then, for any formula $\vphi$ of $\lipc$, 

\begin{equation*}
    w\Vdash_{\mrm{S4},\mc{N}}g(\vphi) \Lra w\Vdash_{\mrm{G},\mc{M}}\vphi 
    \end{equation*}

\end{lemma}

\begin{proof}
    By induction on the complexity of $\vphi$. The atomic case follows by assumption. We display some of the other cases for illustration:
    
    \begin{align*}
      w\Vdash_{\mrm{S4}, \mc{N}}g(\psi\vee \chi) & \Lra w\Vdash_{\mrm{S4}, \mc{N}} g(\psi) \text{ or } w\Vdash_{\mrm{S4}, \mc{N}} g(\chi) & \\
      & \Lra w\Vdash_{\mrm{G}, \mc{M}} \psi \text{ or } w\Vdash_{\mrm{G}, \mc{M}} \chi & \text{by IH}\\
      & \Lra \forall w'\geq_\mc{M}w (w'\Vdash_{\mrm{G}, \mc{M}} \psi )\text{ or } \forall w'\geq_\mc{M}w (w'\Vdash_{\mrm{G}, \mc{M}} \chi) & \text{by Lemma \ref{hereditariness g}}\\
      & \Lra w\Vdash_{\mrm{G}, \mc{M}} \psi\vee\chi\\
    \end{align*}

\begin{align*}
      w\Vdash_{\mrm{S4}, \mc{N}}g(\psi\ra \chi) & \Lra w\Vdash_{\mrm{S4}, \mc{N}} \Box (g(\psi) \ra g(\chi)) & \\
      & \Lra \forall w'\geq_\mc{M}w( w'\Vdash_{\mrm{S4}, \mc{N}} g(\psi) \Ra w'\Vdash_{\mrm{S4}, \mc{N}} g(\chi)) & \\
       & \Lra \forall w'\geq_\mc{M}w( w'\Vdash_{\mrm{G}, \mc{M}} \psi \Ra w'\Vdash_{\mrm{G}, \mc{M}} \chi) & \text{by IH}\\ 
       & \Lra w\Vdash_{\mrm{G}, \mc{M}} \psi\ra \chi & \\
    \end{align*}

\begin{align*}
      w\Vdash_{\mrm{S4}, \mc{N}}g(\forall x \psi) & \Lra w \Vdash_{\mrm{S4}, \mc{N}} \Box \forall vg(\psi(v)) & \\
      &\Lra \forall w'\geq_\mc{M}w, \forall d\in \mc{D}(w')(w'\Vdash_{\mrm{S4}, \mc{N}} g(\psi(c_d)) &\\
      &\Lra \forall w'\geq_\mc{M}w, \forall d\in \mc{D}(w')(w'\Vdash_{\mrm{G}, \mc{M}}\psi(c_d)) & \text{by IH}\\
      & \Lra w\Vdash_{\mrm{G}, \mc{M}} \forall x \psi
    \end{align*}

\end{proof}

\noindent The following lemma obtains in a similar way:

\begin{lemma}\label{atomic sentences 2}
    Let $\mc{N}=\langle W_\mc{N}, \leq_\mc{N}, \mc{D}_\mc{N}, \mc{I}_\mc{N}\rangle$ be a persistent Kripke first-order structure for the language $\lbox$, with $\mc{I}_\mc{N}$ being the triple $\langle \mc{N}_P, \mc{N}_C, \mc{N}_F\rangle$. Now, consider an $\lipc$ model $\mc{M}=\langle W_\mc{M}, \leq_\mc{M}, \mc{D}_\mc{M}, \mc{I}_\mc{M}\rangle$ given by $W_\mc{M}=W_\mc{N}, \leq_\mc{M}=\leq_\mc{N},  \mc{D}_\mc{M}=\mc{D}_\mc{N}$, and $\mc{I}_\mc{M}$ being such that $\mc{M}_C=\mc{N}_C$ and $\mc{M}_F=\mc{N}_F$ and, for all atomic $P(t_1,...,t_n)$:

    \begin{equation*}
    w\Vdash_{\mrm{G},\mc{M}}P(t_1,...,t_n) \Leftrightarrow w\Vdash_{\mrm{S4},\mc{N}} g(P(t_1,...,t_n))
    \end{equation*}

\noindent Then, for any formula $\vphi$ of $\lipc$, 

\begin{equation*}
    w\Vdash_{\mrm{G},\mc{M}}\vphi \Lra w\Vdash_{\mrm{S4},\mc{N}}g(\vphi)
    \end{equation*}

\end{lemma}

\noindent The above two lemmata help in proving that G-countermodels for a formula can be S4-countermodels for that formula, and vice versa:

\begin{lemma}
 If there is a persistent first-order Kripke structure $\mc{M}$ such that $\mc{M}\nvDash_\mrm{G}\vphi$ for $\vphi$ a formula  of $\lipc$, then there is an $\mrm{S4}$ countermodel $\mc{M}'$ such that $\mc{M}'\nvDash_\mrm{S4} g(\vphi)$. 
\end{lemma}

\begin{proof}
Let $\mc{M}=\langle W_\mc{M}, \leq_\mc{M}, \mc{D}_\mc{M}, \mc{I}_\mc{M}\rangle$ be the persistent first-order Kripke countermodel for $\vphi$, with $\mc{I}_\mc{M}$ consisting of the functions $\mc{M}_C, \mc{M}_F, \mc{M}_P$. Then there is some $w\in W_\mc{M}$ s.t. $w\nVdash_{\mrm{G},\mc{M}}\vphi$. Build now an S4 model $\mc{N}$ by letting $W_\mc{N}=W_\mc{M}, \leq_\mc{N}=\leq_\mc{M}, D_\mc{N}=D_\mc{M}, \mc{N}_C=\mc{M}_C, \mc{N}_F=\mc{M}_F$, and letting $v\Vdash_{\mrm{S4},\mc{N}} P(t_1,...,t_n)$ iff $v\Vdash_{\mrm{G},\mc{M}} P(t_1,...,t_n)$, for $P(t_1,...,t_n)$ an atomic formula and $v$ a world in $W_\mc{M}=W_\mc{N}$. Then, for any such atomic $P(t_1,...,t_n)$ and any world $v\in W$, it easily follows by the definition of $g$:

\begin{align*}
    v\Vdash_{\mrm{S4},\mc{N}} g(P(t_1,...,t_n))
    & \Leftrightarrow v\Vdash_{\mrm{G},\mc{M}} P(t_1,...,t_n)
\end{align*}

\noindent Hence, we apply Lemma \ref{atomic sentences} to conclude, from $w\nVdash_{\mrm{G},\mc{M}}\vphi$, that $w\nVdash_{\mrm{S4},\mc{N}}g(\vphi)$, and thus $\mc{N}\nvDash_{\mrm{S4}} g(\vphi)$.

\end{proof}

\begin{lemma}
    If there is an $\mrm{S4}$ model $\mc{M}$ such that $\mc{M}\nvDash_\mrm{S4} g(\vphi)$ for $\vphi$ a formula  of $\lipc$, then there is a persistent first-order Kripke countermodel $\mc{M}'$ such that $\mc{M}'\nvDash_\mrm{G} \vphi$. 
\end{lemma}

\begin{proof}
Let $\mc{N}:=\langle W_\mc{N}, \leq_\mc{N}, D_\mc{N}, \mc{I}_\mc{N}\rangle$ be the S4 countermodel for $\vphi$, with $\mc{I}_\mc{N}$ consisting of the functions $\mc{N}_C, \mc{N}_F, \mc{N}_P$. Then there is $w\in W_\mc{N}$ such that $w\nVdash_{\mrm{S4}, \mc{N}} g(\vphi)$. Now, define the intuitionistic model $\mc{M}=\langle W_\mc{M}, \leq_\mc{M}, D_\mc{M}, \mc{I}_\mc{M}\rangle$ by $W_\mc{M}=W_\mc{N}, \leq_\mc{M}=\leq_\mc{N}, D_\mc{M}=D_\mc{N}, \mc{M}_C=\mc{N}_C, \mc{M}_F=\mc{N}_F$; and with $\mc{M}_P$ such that, for any atomic formula $P(t_1,...,t_n)$, $v\Vdash_{\mrm{G},\mc{M}} P(t_1,...,t_n)$ iff $v\Vdash_{\mrm{S4},\mc{N}} \Box g(P(t_1,...,t_n))$ for any world $v$ in $W_\mc{N}=W_\mc{M}$.

This modification allows us to claim that $\mc{M}$ meets the hereditariness conditions, insofar it is easy to see that $w\Vdash_{\mrm{G},\mc{M}}P(t_1,...,t_n)$ implies $w'\Vdash_{\mrm{G},\mc{M}}P(t_1,...,t_n)$ for all $w'\geq w$. So $\mc{M}$ is indeed a persistent first-order Kripke model. Now, we just apply Lemma \ref{atomic sentences 2} to note that $w\nVdash_{\mrm{G}, \mc{M}} \vphi$
\end{proof}

\noindent From the previous two lemmata, we obtain: 
\begin{lemma}\label{s4 and g}
Let $\vphi$ be a formula of $\lipc$. Then:

\begin{equation*}
    \vDash_\mrm{S4} g(\vphi) \Leftrightarrow \,\vDash_\mrm{G} \vphi
\end{equation*}
\end{lemma}

\noindent Also, the following result can be found, e.g., in \cite[Th.9.2.8]{schwichtenberg_troelstra_1996}:
\begin{lemma}[Schwichtenberg and Troelstra]\label{sch and troel}
\begin{equation*}
    \vdash_\mrm{S4} g(\vphi) \; \Leftrightarrow \; \vdash_\mrm{IPC} \vphi
\end{equation*}
\end{lemma}

\begin{proof}[Proof of Theorem \ref{completeness of G}]
We reason as follows:

\begin{align*}
    \vdash_\mrm{IPC} \vphi & \text{ iff } \vdash_\mrm{S4} g(\vphi) & \text{ by Lemma \ref{sch and troel}}\\
    & \text{ iff } \vDash_\mrm{S4} g(\vphi) & \text{ by strong S4-completeness} \\
    & \text{ iff } \vDash_\mrm{G} \vphi & \text{ by Lemma \ref{s4 and g}}\\
\end{align*}
\end{proof}

\noindent Finally: 

\begin{cortwo}
Let $\vphi$ be a formula of some first-order language $\mc{L}\subseteq \mc{L}_\mrm{IPC}$, and let $\mc
{M}$ range over first-order Kripke models. Then: 

\begin{equation*}
   \Gamma  \vdash_{\mrm{IPC}} \vphi \Ra \, \forall \mc{M}(\mc{M}\vDash_\mrm{G} \Gamma \Ra \mc{M}\vDash_\mrm{G} \vphi)
\end{equation*}
\end{cortwo}

\begin{proof}
If $\Gamma=\varnothing$, that is simply the left-to-right direction of Theorem \ref{completeness of G}. If $\Gamma\neq \varnothing$, since IPC possesses the deduction theorem, it follows that there are $\gamma_1,...,\gamma_n\in \Gamma$ such that $\vdash_\mrm{IPC}\gamma_1\wedge...\wedge\gamma_n\ra \vphi$. Applying again the left-to-right direction of Theorem \ref{completeness of G}, every Kripke model satisfies $\gamma_1\wedge...\wedge\gamma_n\ra \vphi$. So, if $\mc{M}\vDash_\mrm{G}\Gamma$, by the G-forcing clauses it follows that $\mc{M}\vDash_\mrm{G} \vphi$.
\end{proof}

\end{appendices}
\end{document}